\documentclass[12pt]{amsart}
\usepackage{fullpage}
\usepackage{graphicx}
\usepackage{amssymb,amsmath,amsthm, amscd, epsfig}
\usepackage{subfigure}
\input pictex.sty

\theoremstyle{plain}

\newtheorem{theorem}{Theorem}[section]

\theoremstyle{definition}

 \newtheorem*{remark}{Remark}
 \newtheorem{lemma}[theorem]{Lemma}
 \newtheorem{observation}{Observation}

 \newtheorem{defn}{Definition}

 \newcommand{\mc}{\mathcal}

\begin{document}

\title{A variational theory for point defects in patterns}

\author{N. M. Ercolani}%
\address{Dept. of Math., Univ. of Arizona, 617 N. Santa Rita Ave., Tucson, AZ 85721, USA}%
\email{ercolani@math.arizona.edu}%
\author{S.C. Venkataramani}%
\address{Dept. of Math., Univ. of Arizona, 617 N. Santa Rita Ave., Tucson, AZ 85721, USA}%
\email{shankar@math.arizona.edu}


\begin{abstract}
We derive a rigorous scaling law for minimizers in a natural
version of the regularized Cross-Newell model for pattern
formation far from threshold. These energy-minimizing solutions
support defects having the same character as what is seen in
experimental studies of the corresponding physical systems and in
numerical simulations of the microscopic equations that describe
these systems.
\end{abstract}

\thanks{\textit{A Variational Theory for Point Defects in Patterns}}

\maketitle

\section{Introduction}

This paper reports on some recent progress that has been made in
the analytical modeling of defect formation, far from threshold,
in pattern forming physical systems. We will take a moment here to
very briefly sketch the physical and mathematical background that
motivates what is done in this paper.
\smallskip

The relevant class of pattern-forming physical systems to consider
are those in which the spatial physical field can be described as
planar and the first bifurcation from a homogeneous state, having
arbitrary translational symmetry in the plane, produces a striped
pattern which has only a discrete periodic symmetry in one
direction. This \emph{symmetry-breaking} occurs at a critical
threshold; above this threshold the pattern can deform and,
further away, \emph{defects} can form. It is the desire to
understand and model this process of defect formation that
motivates our study.

A good particular example of these kinds of physical systems is a
high Prandtl number Rayleigh-B\'enard convection experiment. The
critical threshold in this example is the critical Rayleigh number
at which fluid convection is initiated from the sub-threshold
homogeneous conducting state. The "striped pattern" here can be
taken to be the horizontal cross-section of the temperature field
at the vertical midpoint of the experimental cell in which
\emph{convection rolls} have formed.

Because of its periodic structure, the striped pattern can be
described in terms of a periodic form function of a \emph{phase},
$\theta = \vec{k}\cdot\vec{x}$, where the magnitude of $\vec{k}$
is the wavenumber of the pattern and the orientation of $\vec{k}$
is perpendicular to the stripes. Here $\vec{x}=(x,y)$ is a
physical point in the plane. Even though the striped pattern will
deform far from threshold, over most of the field (and in
particular away from defects) it can be locally approximated as a
function of a well-defined phase, $\theta(\vec{x})$, for which a
local wavevector can be defined as $\vec{k} = \nabla\theta$ which
differs little from a constant vector unless one varies over
distances on the order of many stripes in the pattern. This
slowly-varying feature of pattern formation far from threshold
motivates the introduction of a \emph{modulational} ansatz in the
microscopic equations describing these physical systems from which
an order parameter equation for the behavior of the phase can be
formally derived. This was originally done by Cross and Newell
\cite{CN}. These equations are variational and from our
perspective it is advantageous to study their solutions by
studying the behavior of the minimizers of the variational
problem. The version of the variational problem that we study
corresponds to the following energy functional on a given domain
$\Omega$ with specified Dirichlet boundary values.

\begin{equation} \label{eq:gl}
{\mathcal{E}}^\mu(\Theta) = \mu \int_\Omega \left(\Delta_{\vec{X}}
 \Theta \right)^2 d\vec{X}
       + \frac{1}{\mu} \int_\Omega (1 - |\nabla_{\vec{X}} \Theta|^2)^2\, d\vec{X}\ ,
\end{equation}
which is expressed in terms of \emph{slow} variables stemming from
the modulational ansatz mentioned above: $\vec{X} = (X,Y) =
\left(\mu x, \mu y\right); \Theta = \frac{\theta}{\mu}$.

We refer to this functional as the \emph{regularized Cross-Newell}
(RCN) \emph{Energy}. It consists of two parts: a non-convex
functional of the gradient (the CN part) plus a quadratic
functional of the Hessian matrix of $\Theta$, which is the
regularizing singular perturbation. Without this regularization,
the CN variational equations admit non-physical caustic formation.
Instead, by studying the limit of minimizers of
${\mathcal{E}}^\mu$ as $\mu \to 0$, one may be able to identify
the formation of a physical defect as a limiting jump
discontinuity or other kind of singularity in the wavevector field
associated to the $\mu$-indexed family of minimizing phase fields.

For more details on what has been rather tersely outlined above,
we refer the reader to \cite{EINP} where analytical results on the
asymptotic limit of minimizers for RCN and their defects in
certain geometries are also derived. See also \cite{ET} where
further refinements and generalizations are developed. We further
mention that the variational problem associated to (\ref{eq:gl})
also arises in other physical contexts (unrelated to pattern
formation) where it is known as the \emph{Aviles-Giga energy}
\cite{AG}.
\medskip

We now turn to the focus of this paper. The kind of defects that are
seen to arise far from threshold are not supported by asymptotic
minimizers of (\ref{eq:gl}) if the class of functions over which one
is varying is restricted to be single-valued phases. In particular,
one can see for purely topological reasons that this restriction rules
out \emph{disclinations} \cite{EINP}.  In \cite{EINP2}, physical,
numerical and experimental arguments are developed which make a strong
case in support of the hypothesis that the correct order parameter
model for the phase in pattern forming systems far from threshold
should come from a variational problem admitting test functions which
are \emph{multi-valued} and in particular \emph{two-valued}. In
physical parlance this is often expressed by saying that the wavefield
$\vec{k}$ should be allowed to be a \emph{director field}; i.e. an
unoriented vector field. One figure (see Fig.~\ref{fig:sh-zip} below)
from \cite{EINP2} will help to crystallize the issue and the focus of
this paper.

\begin{figure}[htbp]

\centerline{\includegraphics[width = 0.9\hsize, angle=90]{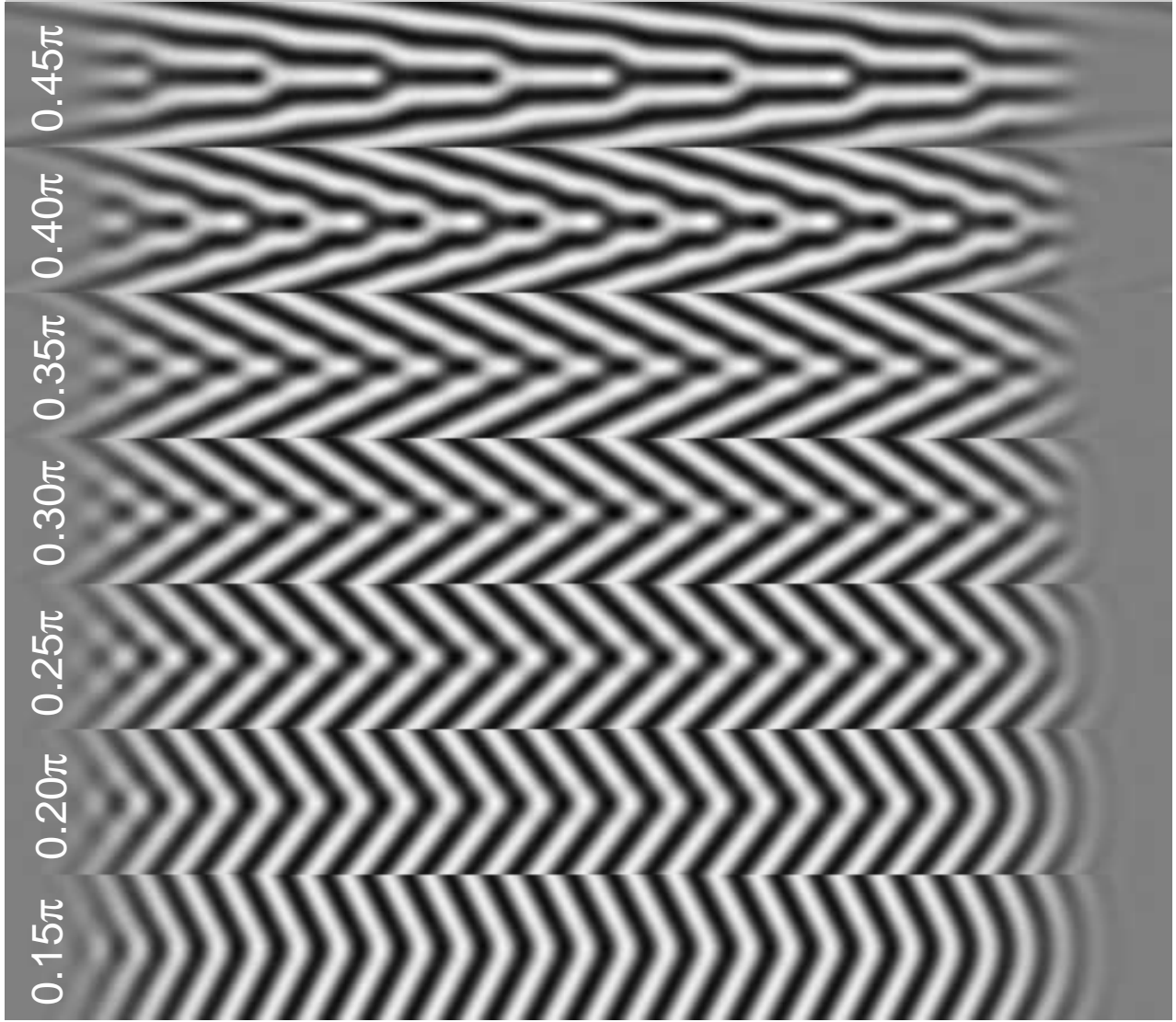}}
\caption{The ``Swift-Hohenberg'' zippers. The patterns are determined
by minimizing the Swift-Hohenberg energy functional for various
choices of the angle $\alpha$ that detemines the slopes of the stripe
patterns as $y \rightarrow \pm \infty$.}
\label{fig:sh-zip}
\end{figure}


This figure shows seven numerical simulations, each done in a
horizontal strip, of a solution to the \emph{Swift-Hohenberg
equation} which is a generic model of microscopic equations for a
pattern forming system. Each of these is run far from theshold but
with differing boundary conditions imposed at the edges. In each
case the boundary conditions impose a constant orientation of the
stripe at the edges such that the normal to the stripe is
$(\cos(\alpha), \sin(\alpha))$ along the top edge and
$(\cos(\alpha), - \sin(\alpha))$ along the bottom edge. The only
thing that changes from one simulation to the next is the value of
$\alpha$ which in the figure is recorded on the left in each
respective cell. The results of \cite{EINP} together with symmetry
considerations establish that for an analogous domain and boundary
values, the asymptotic minimizers of (\ref{eq:gl}), within the
class of single-valued phases, should have the form shown in the
bottom-most cell of Figure (\ref{fig:sh-zip}). That is, they should
have wavevectors very close to $(\cos(\alpha), \sin(\alpha))$ in
the upper region of the cell and very close to $(\cos(\alpha),
-\sin(\alpha))$ in the lower region of the cell with a boundary
layer around the mid-line in which the wavevector transitions
smoothly but rapidly from one state to the other. These minimizers
are dubbed \emph{knee solutions} in \cite{EINP} and in the limit
as $\mu \to 0$, they tend to a configuration in which there is a
sharp jump in the wavevector along the mid-line. This kind of
defect is called a \emph{grain boundary}. In other words, the
theory for (\ref{eq:gl}) with single-valued phases predicts that
the grain boundary should be the limiting defect independent of
the value of $\alpha$. The different result appearing in Figure
(\ref{fig:sh-zip}) was one of the pieces of evidence sited in
\cite{EINP2} to argue the necessity for the larger variational
class of multi-valued phases, even in such simple geometries as
those of Figure (\ref{fig:sh-zip}). In this paper we are going to
carry out a careful analytical study of the RCN variational
problem in exactly this geometry but within a larger class of
two-valued phases. We will firmly establish that the form of the
asymptotic minimizers in this more general model does in fact
depend non-trivially on $\alpha$. In addition, the construction of
test functions in section \ref{sec:u_bound} and the numerical
simulations in section \ref{sec:results} gives some intuitive and
experimental support to the belief that the stable solutions of
the RCN equations qualitatively resemble what is seen in the
Swift-Hohenberg simulations. In \cite{EINP2}, the term
Swift-Hohenberg "zippers" was coined to refer to the problem
studied in Figure (\ref{fig:sh-zip}). In this paper we will be
studying \emph{Cross-Newell zippers}.

\section{Setup} \label{sec:prelim}

We are given an angle $\alpha$ that determines the boundary
conditions on the pattern as $y \rightarrow \pm \infty$ by
$$
\nabla \theta \rightarrow (\cos(\alpha), \pm \sin(\alpha)) \text{ as } y \rightarrow \pm \infty.
$$
Note that this differs from the setup underlying the
Swift-Hohenberg zippers in that the boundary conditions are placed
at $\pm \infty$ in the $y$-direction rather than at finite values
of $y$. This simplifies our technical considerations in that we
don't need to worry about adjusting the location of these
boundaries as $\alpha$ changes.  Also, all of the patterns we want
to consider here are \emph{shift-periodic} in the $x$-direction.
This allows us to reduce our study to domains that are periodic in
$x$. We introduce the (small) parameter $\epsilon = \cos(\alpha)$
and we define the period $l = \pi/\epsilon$. We consider the
following variational problem on the strip $\mc{S}^\epsilon \equiv
\{(x,y) | 0 \leq x < l, y\geq 0\}$:

Minimize $\mc{F}^\epsilon[\theta;a,\delta]$ given by
$$
\mc{F}^\epsilon[\theta;a,\delta] = \iint_{\mc{S}^\epsilon} \left\{[ \Delta \theta]^2 + (1 - |\nabla \theta|^2)^2 \right\} dx dy
$$
over all $a \in [0,1], \delta \in \mathbb{R}$ and $\theta$
satisfying the boundary conditions
\begin{gather}
\theta(x,0) = 0 \quad \text{ for } 0 \leq x < a l; \label{eq:bc} \\
\theta_y(x,0) = 0 \quad \text{ for } a l \leq x < l; \nonumber \\
\theta(x,y) - \epsilon x \text{ is periodic in $x$ with period $l$ for each } y \geq 0; \nonumber \\
\theta(x,y) - \left[\epsilon x + \sqrt{1 - \epsilon^2} y +
\delta\right] \in H^2(\mc{S}^\epsilon). \nonumber
\end{gather}
We take a moment here to explain the considerations that have
motivated the mixed Dirichlet-Neumann boundary conditions here,
the first two boundary conditions in (\ref{eq:bc}) above. We
argued in the introduction that in order to capture the physically
relevant minimizers, the RCN variational problem needed to allow
for multi-valued phases in its admissible class of test functions.
However, the numerical results on the Swift-Hohenberg zippers
suggest that in certain symmetrical geometries the appropriate
multi-valuedness can be introduced in a tractable fashion. Indeed
in the case of the SH zippers we see that the symmetry of the
boundary conditions between the upper and lower edges of the
domain is preserved in the symmetry of all of the exhibited
solutions about the middle horizontal axis; i.e., the reflection
in $y$ about the $y=0$ axis. This suggests that a single-valued
phase could describe the solution in the upper half-plane with the
solution in the lower half-plane given as a symmetric reflection
of that in the upper half-plane about $y=0$.

\begin{figure}[htbp]

\centerline{\includegraphics{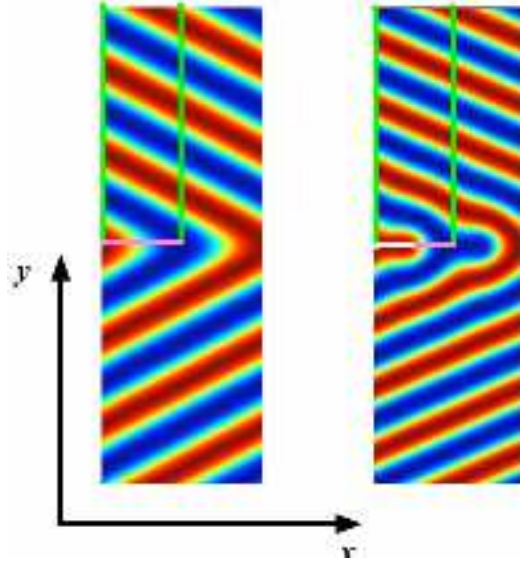}}
\caption{An illustration of the appropriate boundary conditions.}
\label{fig:bcs}
\end{figure}

Figure~\ref{fig:bcs} illustrates two instances of the form that
we expect these zippers to take in the infinite (in $y$) geometry.
The figure on the left illustrates level curves (\textit{stripes}
in the parlance of the introduction) of what we will shortly
define to be a \textit{self-dual knee solution}. This is indeed
symmetric about the mid-axis, which we will take to be the $y=0$
axis; moreover, one can see that its gradient field along $y=0$ is
tangential to this axis. Thus the gradient field in the upper
half-plane is completely symmetrical to that in the lower
half-plane under reflection about $y=0$.

However, for the striped pattern on the right in figure~\ref{fig:bcs},
this is not the case. There are regions, illustrated for example by
the darkened interval along $y=0$, where the gradient field is
tangential to this axis; but, there are other regions, illustrated for
example by the lightened interval along $y=0$, where the gradient
field needs to be perpendicular to this axis. By reflection symmetry
this field will point upwards in the upper half-plane and downward in
the lower half-plane. This cannot be supported by a vector field but
it is allowable for a director field. This indicates that in this
region a two-valued phase is required.

To get at the conditions on the phase itself we observe that
patterns of the type illustrated here are analytically given in
terms of a form function $F$ of the phase $\theta = \theta(x,y)$
such that $F$ is locally periodic of period $2\pi$ in $\theta$ and
such that $F(\theta(x,y))$ is even in $y$ and smooth in $(x,y)$.
In order to allow $\theta$ to be two-valued we also require $F$ to
be even in $\theta$. (An example of a global form function having
these properties is $F = \cos$.) It follows from these
requirements that either $\theta(x,y)$ is even in $y$, in which
case $\theta_y(x,0) = 0$, a Neumann boundary condition; or,
$\theta(x,y)$ is an odd function of $y$ modulo $\pi$, in which
case $\theta_y(x,0) = n\pi$ for some integer $\pi$, a Dirichlet
boundary condition. Thus to realize the pattern on the right in
figure~\ref{fig:bcs} in terms of a single-valued phase in the
upper half-plane, we would need to take the Neumann boundary
condition on the darkened interval and the Dirichlet boundary
condition on the lightened interval. This is what we have done in
(\ref{eq:bc}). For the self-dual knee pattern on the left we would
take the entire boundary condition to be Neumann.
\medskip

The functional $\mc{F}^\epsilon$ is the RCN energy functional but
with the scaling $\mu$ removed. It is appropriate to do this
because the demonstration that the nature of the RCN minimizers
depends on $\alpha$ is independent of this scaling. The first and
the second conditions impose a mixed Dirichlet/Neumann boundary
condition at $y = 0$, the third condition imposes (shifted-)
periodicity in $x$ and the last condition ensures that the test
functions $\theta$ approach the straight parallel roll patterns
$\epsilon x + \sqrt{1 - \epsilon^2} y + \delta$ as $y \rightarrow
\infty$.

Note that the dependence of the functional $\mc{F}^\epsilon$ on
the parameter $\epsilon$ is through the dependence of the domain
$\mc{S}^\epsilon$ and the boundary conditions on $\epsilon$.  The
parameters $a$ and $\delta$ are determined by minimization. The
parameter $a$ is a measure of the fraction of the boundary at $y =
0$ that has a Dirichlet boundary condition, and $\delta$
represents the {\em asymptotic phase shift}, that is the
difference in phases between the test function $\theta$ and the
roll pattern $\hat{\theta}(x,y) = \epsilon x + \sqrt{1 -
\epsilon^2} y$ which satisfies $\hat{\theta}(0,0) = \theta(0,0) =
0$.

The case where $a$ is set to zero is considered in earlier
references \cite{EINP}. The test functions $\theta(x,y)$ satisfy a
pure Neumann boundary condition at $y = 0$ and the minimizers in
this case are the self-dual knee solutions
$$
\theta_{neu}(x,y) = \epsilon x + \log(\cosh(\sqrt{1-\epsilon^2} y)).
$$
These solutions have an asymptotic phase shift of $-\log(2)$ and
the energy of the minimizers in the strip $\mc{S}^\epsilon$ is given
by
\begin{equation}
\mc{F}^\epsilon[\theta_{neu};0,-\log(2)] = \frac{ 4 \pi \sqrt{1-\epsilon^2}}{3 \epsilon}.
\label{eq:chevrons}
\end{equation}

The existence of $(\theta^\epsilon,a^\epsilon,\delta^\epsilon)$
minimizing $\mc{F}^\epsilon$ can be shown from the direct method in
the calculus of variations. We also prove the following results about
the minimizers, and their energy --

\begin{theorem} {\em Upper bound}

  There is a constant $E_0$ such that
  $\mc{F}^\epsilon[\theta^\epsilon;a^\epsilon,\delta^\epsilon] \leq
  E_0$ for all $\epsilon \in (0,1]$.
\label{thm:u_bound}
\end{theorem}

We prove this result in sec.~\ref{sec:u_bound} by exhibiting an
explicit test function satisfying this bound. Note the implication
that the minimizers for sufficiently small $\epsilon$ cannot be the
self-dual solutions, since the energy in Eq.~(\ref{eq:chevrons})
diverges as $\epsilon \rightarrow 0$. Consequently, $a^{\epsilon} > 0$
for sufficiently small $\epsilon$.

\begin{theorem} {\em Lower bound}

  There are constants $E_1 > 0$ and $\epsilon_0 > 0$ such that, even
  for the optimal test function $\theta^\epsilon$ and the optimal
  parameter values $a^\epsilon$ and $\delta^\epsilon$, we have
  $\mc{F}^\epsilon[\theta^\epsilon;a^\epsilon, \delta^\epsilon] \geq
  E_1$ for all $\epsilon \leq \epsilon_0$. Further, there are
  constants $0 < \alpha_1 < \alpha_2$ such that $1 - \alpha_2 \epsilon
  < a^\epsilon < 1 - \alpha_1 \epsilon$ for sufficiently small
  $\epsilon$.
\label{thm:l_bound}
\end{theorem}

We prove this result in sec.~\ref{sec:l_bound}. Combining this
result with the preceding theorem, we obtain a rigorous scaling
law for the energy of the minimizer, and for the quantity $(1-a)$
as $\epsilon \rightarrow 0$. As a corollary to
Theorem~\ref{thm:l_bound}, we find that an $O(1)$ part of the
energy of the minimizer concentrates on the set, $a l \leq x \leq
l, 0 \leq y \leq 1$. This can be interpreted as saying that a
nontrivial part of the energy of the minimizer lives in the region
of the convex-concave disclination pair \cite{EINP2}.

\section{Upper bound} \label{sec:u_bound}

We will first show an upper bound for the enrgy functional
$\mc{F}^{\epsilon}$, uniform in $\epsilon$, by constructing a
family of explicit test function whose energy is uniformly
bounded. The idea for the construction of these test functions
comes from the self-dual ansatz \cite{EINP} which requires that
the energy density of the functional $\mathcal{F}$ should be
\emph{equi-partitioned} between its two terms. Functions
satisfying this ansatz solve the self-dual (resp., anti-self-dual)
equation:
\begin{equation} \label{eq:selfdual}
\Delta\theta =\pm (1 - |\nabla\theta|^2).
\end{equation}
Solutions of this equation can be constructed via the logarithmic
transform
$$
\theta = \pm \log u
$$
which reduces (\ref{eq:selfdual}) to the linear Helmholtz equation
(\ref{selfdual}). We refer the reader to \cite{EINP, ET} for more
background on self-dual reduction.

\subsection{Self-dual test functions for the CN-Zipper
problem}\label{zipper}
\subsubsection{Existence}

We consider the Helmholtz equation in the upper half-plane,
\begin{equation}\label{selfdual}
\Delta u - u = 0
\end{equation}
subject to the mixed boundary conditions
\begin{eqnarray}
u(x,0)   &=& e^{-n\pi} \,\,\, n\ell < x < (n+a)\ell \label{2} \\
u_y(x,0) &=& 0 \,\,\,\,\,\, (n+a)\ell \leq x \leq (n+1)\ell
\label{3}
\end{eqnarray}
and with asymptotic behavior for large $y$ given by const.
$\exp(-\epsilon x - \sqrt{1-\epsilon^2} y)$ where $\ell =
\pi/\epsilon$ and $a \in (0,1)$.

We seek a shift-periodic solution, meaning that we change
variables to $w = e^{\epsilon x} u(x,y)$ and look for periodic
solutions of
\begin{equation}\label{shiftper}
Lw = \Delta w -2\epsilon\partial_x w -(1-\epsilon^2)w=0,
\end{equation}
with boundary conditions of periodicity in $x$ of period $\ell$;
mixed boundary conditions at $y=0$,
\begin{eqnarray}
w(x,0)   &=& e^{\epsilon x} \,\,\,\,\,\, 0 < x < a\ell \label{5}\\
w_y(x,0) &=& 0 \,\,\,\,\,\, a\ell \leq x \leq \ell \label{6};
\end{eqnarray}
and with asymptotic behavior for large $y$ given by const. $\exp(-
\sqrt{1-\epsilon^2} y)$. Given such a $u$, $\theta = -\log u$
would satisfy the boundary conditions (\ref{eq:bc}). (However, for
notational simplicity, in the remainder of this section we will
set $\theta = \log u$.)

We now let $\mc{S}^\epsilon$ denote the half-cylindrical domain,
$\ell$-periodic in $x$ and with $y>0$. The existence of a weak
solution to (\ref{shiftper}) satisfying the above boundary
conditions can be established via the Lax-Milgram theorem with
appropriate energy estimates. However, in order to derive uniform
asymptotic energy estimates (as $\epsilon \to 0$) for the CN
Zipper problem we need to go beyond existence results and try to
construct a more explicit representation of the solution to
(\ref{selfdual}). Unfortunately, at present, the solutions one can
construct using Greens function methods and the like do not yield
sufficient a priori boundary regularity near $y=0$ to control the
asymptotic behavior of the energy in this \textsl{finite} part of
$\mc{S}^\epsilon$. We will therefore instead study solutions of a
self-dual problem with modified boundary conditions (more
precisely, with pure Dirichlet boundary conditions). Subsequently
we will make a local modification of these solutions near the
boundary to produce functions (no longer global self-dual
solutions) whose asymptotic energy we can control \emph{and} which
are valid test functions for the Cross-Newell Zipper problem.

The modified boundary value problem we consider is
(\ref{shiftper}) with (\ref{5}-\ref{6}) replaced by the pure
Dirichlet boundary condition

$$
w(x,0)   = \left\{\begin{array}{c}
   e^{\epsilon x} \,\,\,\,\,\, 0 < x < a\ell \\
   q_a(x)  \,\,\,\,\,\, a\ell \leq x \leq \ell \\
\end{array}\right. \leqno{\begin{array}{c}
  (\ref{5}^\prime)\\
  (\ref{6}^\prime)\\
\end{array}}
$$
where $q_a(x)$ is a function which smoothly interpolates, up
through second derivatives, between $e^{\epsilon x}$ at $x=a\ell$
on the left and $e^{\epsilon x - \pi}$ at $x=\ell$ on the right.
There are clearly many choices for such a function; the precise
choice for our purposes will be made later at the end of
subsection \ref{energy}. By elliptic regularity \cite{Evans}, the
solution to this boundary value problem satisfies $w(x,y) \in
H^2\left(\mc{S}^\epsilon\right)$. In the following sections we
will construct the solutions to this problem and study its
asymptotics relative to the RCN energy $\mathcal{F}^\epsilon$.

\subsubsection{Explicit Construction}

The whole plane Green's function for the Helmholtz equation
(\ref{selfdual}) is explicitly given in terms of the Bessel
potential \cite{Evans}:
\begin{equation}\label{Bessel}
{G}(x,y;\xi,\eta)=\frac{1}{4\pi} \int_0^\infty e^{-t} \frac{dt}{t}
\exp\left( -\frac{1}{4t}\{(x-\xi)^2 + (y-\eta)^2\}\right).
\end{equation}

In terms of this Green's function we can then represent a solution
to (\ref{selfdual}), with asymptotic behavior for large $y$ given
by const. $\exp(-\epsilon x - \sqrt{1-\epsilon^2} |y|)$, as

\begin{eqnarray}\label{soln}
u^\epsilon(x,y) &=& \int_{-\infty}^\infty
\rho^\epsilon(\xi){G}(x,y;\xi,0) d\xi.
\end{eqnarray}

Note that

\begin{eqnarray}\label{dirbase}
  u_y^\epsilon(x,y)  &=& \int_{-\infty}^\infty
\rho^\epsilon(\xi){G}_y(x,y;\xi,0) d\xi\\
\nonumber &=& -\int_{-\infty}^\infty
\rho^\epsilon(\xi){G}_\eta(x,y;\xi,0) d\xi
\end{eqnarray}
solves (\ref{selfdual}) with respect to the standard Dirichlet
boundary condition which equals minus the jump of $u_y^\epsilon$
along the $x$-axis. One may check directly (see (\ref{FT})) that
in fact $\rho^\epsilon(\xi) = -2 u_y^\epsilon(\xi,0)$ almost
everywhere. Integrating (\ref{dirbase}) with respect to $y$ gives

\begin{eqnarray}\label{intdirbase}
  u^\epsilon(x,y)  + f(x) &=& \int_{-\infty}^\infty
\rho^\epsilon(\xi){G}(x,y;\xi,0) d\xi.
\end{eqnarray}
Since both $u^\epsilon(x,y)$ and the RHS of (\ref{intdirbase})
decay as $y\uparrow \infty$, it follows that $f(x)\equiv 0$. This
is consistent with the ansatz (\ref{soln}), taking
$\rho^\epsilon(\xi)$ to be the jump in the normal derivative of
$u^\epsilon$ along $y=0$.
\medskip

We make the following shift-periodic ansatz for $\rho^\epsilon$,

$$
\rho^\epsilon(\xi+\ell)e^{\epsilon(\xi+\ell)} =
\rho^\epsilon(\xi)e^{\epsilon\xi}.
$$
With this one can expand out (\ref{soln}) more explicitly as

\noindent \,\,$u^\epsilon(x,y)=$
\begin{eqnarray}
   &=& \frac{1}{4\pi}\sum_{n \in
\textbf{Z}} \int_0^\infty\frac{dt}{t} e^{-(t+\frac{y^2}{4t})}
\int_{n\ell}^{(n+1)\ell} \rho^\epsilon(\xi)
\exp\left(\frac{(x-\xi)^2}{-4t}\right) d\xi  \label{11}\\
  &=& \frac{1}{4\pi}\sum_{n \in \textbf{Z}}
\int_0^\infty\frac{dt}{t} e^{-(t+\frac{y^2}{4t})} \int_{0}^{\ell}
e^{-n\pi}\rho^\epsilon(\xi)
\exp\left(\frac{(x-(\xi + n\ell))^2}{-4t}\right) d\xi \label{12}\\
&=& \frac{1}{4\pi} \int_0^\infty\frac{dt}{t}
e^{-(t+\frac{y^2}{4t})} \int_{0}^{\ell}\rho^\epsilon(\xi) \sum_{n
\in \textbf{Z}}e^{-n\pi} \exp\left(\frac{(x-(\xi +
n\ell))^2}{-4t}\right) d\xi
\label{13}\\
&=& \frac{e^{-\epsilon x}}{4\pi} \int_0^\infty\frac{dt}{t}
e^{-((1-\epsilon^2)t+\frac{y^2}{4t})}
\int_{0}^{\ell}d\xi\rho^\epsilon(\xi)e^{\epsilon\xi}\sum_{n \in
\textbf{Z}} \exp\left(\frac{((x-2\epsilon t)-(\xi +
n\ell))^2}{-4t}\right)
\label{14}\\
&=& \frac{e^{-\epsilon x}}{\sqrt{4\pi}}
\int_0^\infty\frac{dt}{t^\frac{1}{2}}
e^{-((1-\epsilon^2)t+\frac{y^2}{4t})}
\frac{1}{\ell}\int_{0}^{\ell} \rho^\epsilon(\xi) e^{\epsilon \xi}
\,\,\vartheta_3\left(\frac{-(x - \xi) + 2\epsilon t}{\ell},
\frac{-4\pi t}{\ell^2}\right)
  d\xi \label{15}
\end{eqnarray}
In (\ref{11}) we have interchanged the order of integration which
is justified by Tonelli's Theorem; in (\ref{12}) we've made the
substitution $\xi = \xi_n + n\ell$ and in (\ref{13}) we've
commuted the sum past the integrals which is justified by monotone
convergence--all terms in the series are positive and hence the
partial sums are monotonic. In (\ref{14}) we write each summand as
a single exponential and then appropriately complete the square in
each exponent. Finally in (\ref{15}) we apply Jacobi's identity
\cite{WW}. Here $\vartheta_3$ is one of the Jacobi theta
functions, in this setting explicitly given as
\begin{eqnarray} \label{Jacobi}
\vartheta_3\left(\frac{- x + 2\epsilon t}{\ell}, \frac{-4\pi
t}{\ell^2}\right) &=& 1 + 2\sum_{n=1}^\infty
e^{-\left(\frac{2\pi}{\ell}\right)^2 n^2 t} \cos\left(\frac{2\pi
n}{\ell}\left(x-\frac{2\pi t}{\ell}\right) \right)
\end{eqnarray}

Finally, from (\ref{15}) we can express our candidate for the
solution to (\ref{shiftper}), ($\ref{5}^\prime - \ref{6}^\prime$)
as
\begin{eqnarray}
{w^\epsilon}(x,y) = \frac{1}{\sqrt{4\pi}}
\int_0^\infty\frac{dt}{t^\frac{1}{2}}
e^{-((1-\epsilon^2)t+\frac{y^2}{4t})}
\frac{1}{\ell}\int_{0}^{\ell} p^\epsilon(\xi)
\,\,\vartheta_3\left(\frac{-(x - \xi) + 2\epsilon t}{\ell},
\frac{-4\pi t}{\ell^2}\right)
  d\xi,\label{shiftpersoln}
\end{eqnarray}
where $p^\epsilon(\xi) = \rho^\epsilon(\xi) e^{\epsilon \xi}$.

\subsubsection{Data Characterization, periodized and in
Fourier Space}

 From the previous sections we have that $p^\epsilon(\xi)$ is
periodic of period $\ell$; also ${w^\epsilon}(x,y)$ is periodic in
$x$ of period $\ell$ and $=e^{\epsilon x}$ along $(0, a\ell)$ when
$y=0$.

Moreover, taking the Fourier transform of (\ref{shiftpersoln}) one
finds that the Fourier coefficients, in $x$, must satisfy
\begin{eqnarray} \label{FT}
   \{\widehat{w^\epsilon(x,y)}\}(n,y) &=& \frac{1}{2}\frac{1}{\sqrt{1+\epsilon^2(2n+i)^2}}{\{\widehat{p^\epsilon(\xi)}\}}(n)e^{-\sqrt{1+\epsilon^2(2n+i)^2}y}
\end{eqnarray}
for each value of $y$. Taking the limit as $y\to 0$ on both sides
of (\ref{FT}) gives
\begin{eqnarray} \label{FTsoln}
   \{\widehat{w^\epsilon(x,0)}\}(n) &=&
\frac{1}{2}\frac{1}{\sqrt{1+\epsilon^2(2n+i)^2}}{\{\widehat{p^\epsilon(\xi)}\}}(n).
\end{eqnarray}
This is a determining conditions for $p^\epsilon(\xi)$. We note
that differentiating (\ref{FT}) with respect to $y$ and setting $y
= 0$ demonstrates that $p^\epsilon(x) = - 2 w^\epsilon_y(x,0)$, at
least in the $L^2$ sense.

Since $w^\epsilon(x,0) \in H^{2}(S^1)$, it follows, by comparison,
that $2\sqrt{1+ \epsilon^2(2n+i)^2}\widehat{w^\epsilon}(n)\in
h^1(\mathbb{Z})$. Given this we can now define
\begin{equation} \label{Besspot}
p^\epsilon(x) \doteq \left\{2\sqrt{1+
\epsilon^2(2n+i)^2}\widehat{w^\epsilon}(n)\right\}^\vee(x)
\end{equation}
which characterizes $p^\epsilon$ as an element of $H^1(S^1)$. It
follows from Sobolev's lemma \cite{Evans} that $p^\epsilon$ can be
taken to be continuous. This last observation also justifies the
existence of the Fourier coefficients
$\{\widehat{p^\epsilon}\}(n)$ that were formally introduced in
(\ref{FT}).

\subsubsection{Large $y$ asymptotics}\label{largey} We now
determine the large $y$ asymptotics of (\ref{soln}). By
(\ref{FT}), $w^\epsilon$ has a Fourier representation given by
\begin{eqnarray*}
{w^\epsilon}(x,y) &=& \sum_{n \in \mathbb{Z}}
\widehat{w^\epsilon}(n) e^{-\sqrt{1+\epsilon^2(2n+i)^2}y}
e^{\frac{2\pi i nx}{\ell}}\\
&=& \widehat{w^\epsilon}(0) e^{-\sqrt{1-\epsilon^2}y} +
\mathcal{O}\left(e^{-2 \sqrt{1 + 3\epsilon^2} y}\right).
\end{eqnarray*}
(We note that for large $y$ this series converges uniformly to a
smooth, in fact real-analytic, function of $x$.) Moreover,
$w^\epsilon(0) = \frac{1}{\ell}\int_{0}^{\ell}w^\epsilon(x,0) dx$
is non-zero since by the maximum principle \cite{Evans} applied to
the elliptic PDE (\ref{shiftper}) on the cylinder $[0,\ell] \times
\left(-\infty, \infty\right)$, the integrand, $w^\epsilon(x,0)$,
is non-negative and in fact, by ($\ref{5}^\prime-\ref{6}^\prime$),
non-vanishing on $[0,\ell]$ (the definition of $q_a$ which we give
later will insure that this is so).

\subsection{Energy Estimates}\label{energy}

We will now try to show that the regularized Cross-Newell energy
of $\theta(x,y) = \log u(x,y)$ is uniformly bounded in $\epsilon$.
This would establish a uniform (in $\epsilon$) upper bound for the
energy minimizers. Recall that the energy is calculated by
integrating the energy density over the domain $\mc{S}^\epsilon$.
Making this estimate breaks naturally into the consideration of
two regions: $[0, \ell] \times \{y\geq M_\epsilon\}$ and $[0,
\ell] \times \{y < M_\epsilon\}$ where $M_\epsilon$ is to be
determined.

We remark that the so-called "knee solution" of the self-dual
equation provides an upper bound for the energy for values of
$\epsilon$ bounded away from zero. So we only need to be concerned
with small values of $\epsilon$. Since $u^\epsilon(x,y) =
e^{-\epsilon x}w^\epsilon(x,y)$ solves the Helmholtz equation, it
will suffice to bound the density
$(1-|\nabla\theta^\epsilon|^2)^2$ (since the integral of this
density equals that of $(\Delta \theta^\epsilon)^2$ for self-dual
solutions).

\subsubsection{Estimates in $[0, \ell] \times \{y\geq M_\epsilon\}$}
We begin by considering the domain for large $y$. Since

\begin{eqnarray*}
&&\nabla\theta^\epsilon(x,y) = \frac{\nabla
u^\epsilon}{u^\epsilon}(x,y) = \left(
\begin{array}{c}
  -\epsilon  \\
   0 \\
\end{array}
\right) + \frac{\nabla w^\epsilon}{w^\epsilon}(x,y),
\end{eqnarray*}
we may reduce our considerations to studying the asymptotics of
$w^\epsilon$ and its first derivatives. It will be convenient to
replace the convolution integral in (\ref{shiftpersoln}) by the
Fourier series whose coefficients are the product of the Fourier
coefficients of $p^\epsilon$ and the $\vartheta_3$ series. This
results in the following alternative representation of
$w^\epsilon$:

\begin{eqnarray*}
{w^\epsilon}(x,y) &=& \frac{1}{\sqrt{4\pi}} \sum_{n\in
\mathbb{Z}}\int_0^\infty\frac{dt}{t^\frac{1}{2}}
e^{-
\left((1+\epsilon^2(2n+i)^2)t+\frac{y^2}{4t}\right)}\widehat{p^\epsilon}
(n)e^{\frac{2\pi
i n x}{\ell}}.
\end{eqnarray*}
With the change of variables,
$$
s = \frac{t}{y}
$$
this representation takes the form
\begin{eqnarray}\label{asympw}
{w^\epsilon}(x,y) &=& \sqrt{\frac{y}{4\pi}} \sum_{n\in
\mathbb{Z}}\int_0^\infty\frac{ds}{s^\frac{1}{2}}
e^{-\frac{y}{4}\left(\frac{s}{s_n^2} +
\frac{1}{s}\right)}\widehat{p^\epsilon}(n)e^{\frac{2\pi i n
x}{\ell}},
\end{eqnarray}
where $s_n = \frac{1}{2\sqrt{1 + \epsilon^2 (2n+i)^2}}$. The
critical point of the exponent is $s=s_n$ and the expansion of the
exponent in the $nth$ term of the series near this critical point
has the form
\begin{eqnarray*}
\frac{s}{s_n^2} + \frac{1}{s} &=& \frac{2}{s_n}\left(1 +
\frac{(s-s_n)^2}{s_n^2}+
\mathcal{O}\left(\frac{(s-s_n)^3}{s_n^3}\right)\right).
\end{eqnarray*}
An asymptotic expansion in large $y$ may be developed for the
integral in each term of the series (\ref{asympw}) by the method
of Laplace. By the uniform convergence of the series (for large
$y$), the asymptotic expansion of the series is equivalent to the
sum of the asymptotic expansions from each term. We implement this
strategy to find the leading order, large $y$ behavior, and next
corrections, for $w^\epsilon$, $w_x^\epsilon$ and $w_y^\epsilon$:

\begin{eqnarray*}
   w^\epsilon(x,y) &=& \sqrt{\frac{y}{4\pi}} \sum_{n\in
\mathbb{Z}}s_n^\frac{1}{2} e^{-\frac{y}{s_n}}\int_{-1}^\infty
\frac{dz}{(1+z)^\frac{1}{2}}
e^{-
\frac{y}{2s_n}\left(z^2+\mathcal{O}(z^3)\right)}\widehat{p^\epsilon}(n)e
^{\frac{2\pi
i n
x}{\ell}}, \\
   w_x^\epsilon(x,y) &=& -2 i \epsilon \sqrt{\frac{y}{4\pi}} \sum_{n\ne
0} s_n^\frac{1}{2} e^{-\frac{y}{s_n}}\int_{-1}^\infty
\frac{dz}{(1+z)^\frac{1}{2}}
e^{-\frac{y}{2s_n}\left(z^2+\mathcal{O}(z^3)\right)}n
\widehat{p^\epsilon}(n)e^{\frac{2\pi i n
x}{\ell}}, \\
   w_y^\epsilon(x,y) &=& \frac{1}{2y}w^\epsilon -
\frac{1}{2}\sqrt{\frac{y}{4\pi}} \sum_{n\in
\mathbb{Z}}s_n^{-\frac{1}{2}} e^{-\frac{y}{s_n}}\int_{-1}^\infty
\frac{dz}{(1+z)^\frac{1}{2}}
e^{-
\frac{y}{2s_n}\left(z^2+\mathcal{O}(z^3)\right)}\widehat{p^\epsilon}(n)e
^{\frac{2\pi
i n x}{\ell}},
\end{eqnarray*}
where in the $n^{th}$ term of each series, $z=\frac{s-s_n}{s_n}$,
respectively. We can now apply Laplace's method to each term and
then observe that the dominant contributions for large $y$ come
from the $0, +1, -1$ Fourier modes. Retaining just these we derive
the following asymptotic behavior for $\nabla \log w^\epsilon$:

\begin{eqnarray*}
\frac{w_x^\epsilon}{w^\epsilon}(x,y) &=&
\frac{-4\epsilon\sqrt{1-\epsilon^2}}{\widehat{p^\epsilon}(0)}e^{-2\epsilon^2y}\Im\left(\widehat{p^\epsilon}(1)
e^{-2i \epsilon^2 y}\right) = \mathcal{O}\left(\epsilon e^{-2\epsilon^2 y}\right)\\
\frac{w_y^\epsilon}{w^\epsilon}(x,y) &=& \frac{1}{2y}
-\sqrt{1-\epsilon^2}\frac{1 +
\Re\left(\frac{\widehat{p^\epsilon}(1)}{\widehat{p^\epsilon}(0)}
e^{2i(\epsilon x - \epsilon^2 y)}\right)e^{-2\epsilon^2 y}
+\mathcal{O}\left(\epsilon^2e^{-2\epsilon^2 y}\right)}{1 +
\Re\left(\frac{\widehat{p^\epsilon}(1)}{\widehat{p^\epsilon}(0)}e^{2i(\epsilon
x - \epsilon^2 y)}\right)e^{-2\epsilon^2 y}
+\mathcal{O}\left(\epsilon^2e^{-2\epsilon^2 y}\right)}\\
  &=&
-\sqrt{1-\epsilon^2} + \frac{1}{2y} +\mathcal{O}\left(\epsilon^2
e^{-2\epsilon^2 y}\right).
\end{eqnarray*}

Based on these asymptotics we can now estimate the energy in the
large $y$ domain.
\begin{eqnarray*}
   \nabla \theta^\epsilon &=& \left(
\begin{array}{c}
  -\epsilon \\
   -\sqrt{1-\epsilon^2} \\
\end{array}
\right) + \left(
\begin{array}{c}
  \mathcal{O}\left(\epsilon e^{-2\epsilon^2 y} \right)  \\
  \mathcal{O}\left(\frac{1}{y} + \epsilon^2 e^{-2\epsilon^2 y} \right) \\
\end{array}
\right)
\end{eqnarray*}
from which it follows that
\begin{eqnarray*}
   |\nabla \theta^\epsilon|^2 &=& 1 + \mathcal{O}\left(\epsilon^2
e^{-2\epsilon^2 y}
   \right) + \mathcal{O}\left(\frac{1}{y}\right)
\end{eqnarray*}
and so the energy density
\begin{eqnarray*}
   \left(1- |\nabla \theta^\epsilon|^2\right)^2 &=&
\mathcal{O}\left(\epsilon^4 e^{-4\epsilon^2 y}
   \right) + \mathcal{O}\left(\frac{\epsilon^2}{y} e^{-2\epsilon^2 y}
   \right) + \mathcal{O}\left(\frac{1}{y^2}\right).
\end{eqnarray*}
>From this it follows that the "large $y$" part of the total energy
is bounded as
$$
\mathcal{F}^\epsilon_{y\geq M_\epsilon} \lesssim \frac{1}{\epsilon
M_\epsilon}.
$$
Thus, if we take $M_\epsilon = c/\epsilon$, this part of the total
energy will remain finite as $\epsilon \to 0$.
\medskip

\subsubsection{Estimates in $[0,\ell] \times \left\{y <M_\epsilon\right\}$}
We next turn to consideration of the energy density in the
\textsl{finite} part of the domain. To facilitate this
consideration we will sometimes make the uniformizing change of
variables $z = \epsilon \xi$ and $h = \epsilon x$ in the Jacobi
theta function (\ref{Jacobi}):
\begin{eqnarray}
\vartheta_3\left(\frac{- (x-\xi) + 2\epsilon t}{\ell}, \frac{-4\pi
t}{\ell^2}\right)=\vartheta_3\left(\frac{-(h - z) + 2\epsilon^2
t}{\pi}, \frac{-4\epsilon^2 t}{\pi}\right) \label{perJacobi}
\end{eqnarray}
In what follows we will assume that $a$ is chosen to depend on
$\epsilon$ in such a way that $1-a^\epsilon =
\mathcal{O}(\epsilon)$.
\medskip

We will make use here of the following single-layer potential
counterpart of the double-layer potential representation
(\ref{shiftpersoln}), which in fact can be deduced directly from a
change of variables in (\ref{asympw}):

$\noindent {w^\epsilon}(x,y)=$
\begin{eqnarray}
 &=& \frac{-y}{\sqrt{4\pi}}
\int_0^\infty\frac{dt}{t^\frac{3}{2}}
e^{-((1-\epsilon^2)t+\frac{y^2}{4t})}
\frac{1}{\ell}\int_{0}^{\ell} w^\epsilon(\xi,0)
\,\,\vartheta_3\left(\frac{-(x - \xi) + 2\epsilon t}{\ell},
\frac{-4\pi t}{\ell^2}\right)
  d\xi,\label{shiftpersoln2}
\end{eqnarray}

We study the asymptotic behavior of the convolution integral in
(\ref{shiftpersoln2}) for $x \in (0, a\ell)$ and for
\textit{times} $t$ of order less than $1/\epsilon$:

\begin{eqnarray}\label{convoest}
&&\frac{1}{\ell}\int_{0}^{\ell} w^\epsilon(\xi) \,\,
\vartheta_3\left(\frac{- (x-\xi) + 2\epsilon t}{\ell}, \frac{-4\pi
t}{\ell^2}\right)d\xi \\
\nonumber &=&\frac{1}{\ell}\int_{0}^{\ell} e^{\epsilon \xi} \,\,
\vartheta_3\left(\frac{- (x-\xi) + 2\epsilon t}{\ell}, \frac{-4\pi
t}{\ell^2}\right)d\xi\\
 \nonumber &+& \frac{1}{\ell}\int_{a\ell}^{\ell}
\left(w^\epsilon(\xi) - e^{\epsilon \xi}\right) \,\,
\vartheta_3\left(\frac{- (x-\xi) + 2\epsilon t}{\ell}, \frac{-4\pi
t}{\ell^2}\right)d\xi \\
\nonumber &=& \frac{1}{\pi}\int_{0}^{\pi} e^{z}
\,\,\vartheta_3\left(\frac{-(h - z) + 2\epsilon^2 t}{\pi},
\frac{-4\epsilon^2 t}{\pi}\right) dz\\
\nonumber &+& \frac{1}{\pi}\int_{0}^{\pi}
\left(q_a(\frac{z}{\epsilon}) - e^{z}\right) \,\,
\vartheta_3\left(\frac{-(h - z) +
2\epsilon^2 t}{\pi},\frac{-4\epsilon^2 t}{\pi}\right) dz \\
\nonumber &=& e^{\epsilon x} + o(\epsilon),
\end{eqnarray}
where in the third line above, the form of the integrals follows
from making the change of variables as in (\ref{perJacobi}). In
the second integral we smoothly extend $q_a(\frac{z}{\epsilon}) -
e^z$ to be zero on $(0,a\pi)$. The final line follows for $t$ of
order less than $1/\epsilon$ because in this regime the Jacobi
theta function inside the convolution behaves as a \textit{Dirac
comb} as $\epsilon \to 0$. The second term has this asymptotic
behavior because $h \in (0,a\pi)$ and the support of
$q_a(\frac{z}{\epsilon}) - e^z$ is complementary to this interval,
so that this integral decays exponentially to zero with
$\epsilon$, as with a Dirac sequence away form its support.

Based on (\ref{convoest}) we can estimate $w^\epsilon$ as
\begin{eqnarray} \label{w-limit}
{w}^\epsilon(h,y) &=& \nonumber \frac{-y}{\sqrt{4\pi}}
\left[\int_0^{1/\epsilon}\frac{dt}{t^\frac{3}{2}}
e^{-((1-\epsilon^2)t+\frac{y^2}{4t})} \left(e^{\epsilon x} +
{o}(\epsilon)\right)\right] +
\mathcal{O}\left(e^{-1/\epsilon}\right)\\
&=& e^{\epsilon x} e^{-\sqrt{1-\epsilon^2}y}+ o(\epsilon).
\end{eqnarray}
The evaluation of the previous integral may be deduced from a
basic Bessel identity (see \cite{AS} 9.6.23).

In order to estimate $\nabla\theta^\epsilon$, we also need to
estimate the $x$ and $y$ derivatives of $w^\epsilon(x,y)$. To this
end we first consider the $x$-derivative of the internal
convolution integral which equals
\begin{eqnarray}
&&\partial_x \frac{1}{\pi}\int_{0}^{\pi} \left(
q_a(\frac{z}{\epsilon}) - e^z \right)
\,\,\vartheta_3\left(\frac{-(h - z) +
2\epsilon^2 t}{\pi}, \frac{-4\epsilon^2 t}{\pi}\right) dz \\
&=& \epsilon \frac{1}{\pi}\int_{0}^{\pi} \left(
q_a(\frac{z}{\epsilon}) - e^z \right) \,\,\partial_h
\vartheta_3\left(\frac{-(h - z) + 2\epsilon^2 t}{\pi},
\frac{-4\epsilon^2 t}{\pi}\right) dz \\
  &=& \nonumber -
\epsilon\frac{1}{\pi}\int_{0}^{\pi} \left( q_a(\frac{z}{\epsilon})
- e^z \right) \,\,\partial_z\,\vartheta_3\left(\frac{-(h - z) +
2\epsilon^2 t}{\pi}, \frac{-4\epsilon^2 t}{\pi}\right) dz.
\end{eqnarray}
Integrating by parts, the above derivative may be rewritten as
\begin{eqnarray}
\label{xderiv}&& \frac{1}{\pi}\int_{0}^{\pi} \epsilon\partial_z
\left( q_a(\frac{z}{\epsilon}) - e^z \right)\,
\,\,\vartheta_3\left(\frac{-(h - z) +
2\epsilon^2 t}{\pi}, \frac{-4\epsilon^2 t}{\pi}\right) dz\\
&=& \nonumber o(\epsilon)\,\, \mbox{for}\,\,\, h \in (0,a\pi),
\end{eqnarray}
as for the second integral in the last line of (\ref{convoest}).
Thus,
\begin{eqnarray}
\nonumber \frac{u^\epsilon_x}{u^\epsilon} &=& -\epsilon +
\frac{\frac{y}{\sqrt{4\pi}} \int_0^\infty\frac{dt}{t^\frac{3}{2}}
e^{-((1-\epsilon^2)t+\frac{y^2}{4t})} \frac{1}{\pi}\int_{0}^{\pi}
  \epsilon\partial_z \left( q_a(\frac{z}{\epsilon}) -
e^z \right)\,\,\vartheta_3\left(\frac{-(h - z) + 2\epsilon^2
t}{\pi}, \frac{-4\epsilon^2 t}{\pi}\right)
  dz}{\frac{y}{\sqrt{4\pi}}
\int_0^\infty\frac{dt}{t^\frac{3}{2}}
e^{-((1-\epsilon^2)t+\frac{y^2}{4t})} \frac{1}{\pi}\int_{0}^{\pi}
  \left( q_a(\frac{z}{\epsilon}) -
e^z \right)\,\,\vartheta_3\left(\frac{-(h - z) + 2\epsilon^2
t}{\pi}, \frac{-4\epsilon^2 t}{\pi}\right)
  dz}\\
\label{xlogd}&=&  -\epsilon + \frac{o(\epsilon)}{e^{\epsilon x} +
o(\epsilon)} = \mathcal{O}(\epsilon)
\end{eqnarray}
For the $y$ logarithmic derivative we have

\begin{eqnarray}
\nonumber \frac{{u}^\epsilon_{y}}{{u}^\epsilon} &=& \frac{1}{y} -
\frac{\frac{y}{2}\int_0^\infty\frac{dt}{t^\frac{5}{2}}
e^{-((1-\epsilon^2)t+\frac{y^2}{4t})}
  \frac{1}{\ell}\int_{0}^{\ell} w^{\epsilon}(\xi,0)
\vartheta_3\left(\frac{-(x - \xi) + 2\epsilon t}{\ell},
\frac{-4\pi t}{\ell^2}\right)
  d\xi}{\int_0^\infty\frac{dt}{t^\frac{3}{2}}
e^{-((1-\epsilon^2)t+\frac{y^2}{4t})}\frac{1}{\ell}\int_{0}^\ell
w^{\epsilon}(\xi,0) \vartheta_3\left(\frac{-(x- \xi) + 2\epsilon
t}{\ell}, \frac{-4\pi t}{\ell^2}\right)
  d\xi } \\
  \nonumber &=& \frac{1}{y} -
\frac{\frac{y}{2}\int_0^\infty\frac{dt}{t^\frac{5}{2}}
e^{-((1-\epsilon^2)t+\frac{y^2}{4t})}
  \left(e^{\epsilon x} +
o(\epsilon)\right)}{\int_0^\infty\frac{dt}{t^\frac{3}{2}}
e^{-((1-\epsilon^2)t+\frac{y^2}{4t})}\left(e^{\epsilon x} +
o(\epsilon)\right)}\\
\label{ylogd} &=& \frac{1}{y} -
\frac{K_{-\frac{3}{2}}(y)}{K_{-\frac{1}{2}}(y)} + o(\epsilon) = -1
+ o(\epsilon).
\end{eqnarray}
The last equivalence follows from a Bessel recurrence identity
\cite{AS}, formula 9.6.26, together with formula 9.6.6.

Thus we finally have

\begin{eqnarray}
\nabla \theta^\epsilon = \left(
\begin{array}{c}
   0 \\
   -1 \\
\end{array}
\right) + \mathcal{O}(\epsilon)
\end{eqnarray}
and hence $(1-|\nabla\theta^\epsilon|^2)^2 =
\mathcal{O}(\epsilon^2)$. Since the domain $[0,a\ell] \times
\left\{y < M_\epsilon\right\}$ has dimensions $1/\epsilon \times
1/\epsilon$, the total energy in this region is also
asymptotically finite.

\subsubsection{Modification of the Self-dual Test Funciton} It
remains to estimate the energy in the region $[a\ell, \ell] \times
\left\{y < M_\epsilon\right\}$ which has dimensions
$\mathcal{O}(1) \times 1/\epsilon$. The question of the finiteness
of the energy of the $\theta^\epsilon$ we have been considering in
this region is beside the point for general purpose of but, this
self-dual solution does not satisfy the boundary condition
(\ref{6}) in this region.

As stated earlier we are going to modify the self-dual test
function in this region so that the boundary condition (\ref{6})
is satisfied. To that end we fix a small value of $\delta$ and let
$\mathbf{B}(\delta)$ denote the $\delta$-neighborhood of $[a\ell,
\ell]$ in $\mathcal{S}^\epsilon$. We modify $w^\epsilon$ in this
neighborhood as follows. Define
\begin{eqnarray} \label{finalsoln}
   \widetilde{w^\epsilon}(x,y) &=& \phi_1(x,y) w^\epsilon(x,y) +
\phi_2(x,y) w_2(x,y)
\end{eqnarray}
where $\{\phi_1,\phi_2\}$ is a partition of unity subordinate to
the cover of $\mathcal{S}^\epsilon$ given by
\begin{eqnarray*}
   U_1 &=& \mathcal{S}^\epsilon\backslash \mathbf{B}(\delta/2) \\
   U_2 &=& \mathbf{B}(\delta)
\end{eqnarray*}
and $w_2(x,y) = w^\epsilon(x,0)\cosh\left( y\right)$, where
$w^\epsilon(x,0)$ here is defined as in $(\ref{5}^\prime -
\ref{6}^\prime)$. One has
\begin{eqnarray*}
   \phi_1 &=& \left\{\begin{array}{cc}
     1 & U_1\backslash  \mathbf{B}(\delta)\\
     0 & \mathbf{B}(\delta/2) \\
\end{array}\right. \\
   \phi_2 &=& \left\{\begin{array}{cc}
     1 & \mathbf{B}(\delta/2) \\
     0 & U_1\backslash  \mathbf{B}(\delta)\\
    \end{array} \right.
\end{eqnarray*}
and $\phi_1 + \phi_2 \equiv 1$.

It is straightforward to check that $\widetilde{w^\epsilon}(x,y)$
satisfies the boundary conditions (\ref{5}) and (\ref{6}):
\begin{eqnarray*}
   \lim_{y\to 0} \widetilde{w^\epsilon}(x,y) &=& \phi_1(x,0)
w^\epsilon(x,0) + \phi_2(x,0)
   w^\epsilon(x,0)\\
   &=& \left(\phi_1(x,0)+ \phi_2(x,0)\right) w^\epsilon(x,0)\\
   &=& w^\epsilon(x,0)\\
   &=& e^{\epsilon x}
\end{eqnarray*}
for $x\in [0,a\ell]$.

For $x \in [a\ell, \ell]$,
\begin{eqnarray*}
   \lim_{y\to 0} \widetilde{w^\epsilon}_y(x,y) &=& \left(\phi_{1y}(x,0)+
\phi_{2y}(x,0)\right) w^\epsilon(x,0)
    + \phi_2(x,0) w^\epsilon(x,0)\sinh(0) \\
    &=& \left(\phi_1 + \phi_2\right)_y (x,0) w^\epsilon(x,0) + 0 \\
    &=& 0.
\end{eqnarray*}
Thus, $\log \widetilde{w^\epsilon}$ is an admissible test function
for the regularized Cross-Newell variational problem. We can now
estimate the energy of this test function in $\mathbf{B}(\delta)$.
The energy density in this region is bounded and therefore the
energy in $\mathbf{B}(\delta)$ is finite.

\subsubsection{Estimates for the "outer" solution in $\left(
[a\ell,\ell] \times \left\{y < M_\epsilon\right\}\right)$} It
remains to estimate the energy in $\left( [a\ell,\ell] \times
\left\{y < M_\epsilon\right\}\right) \backslash \mathbf{B}(\delta)
$. To proceed with this we will need a more specific definition of
$q_a$ which we now give.

Note first that by our assumption that $1 - a^\epsilon =
\mathcal{O}(\epsilon)$, the interval $[a\ell,\ell]$ remains of
size $\mathcal{O}(1)$ for arbitrarily small values of $\epsilon$.
We will now further pin this down by setting $1-a^\epsilon = c
\epsilon$ for a value of $c$ that is fixed, independent of
$\epsilon$. Consequently, $[a\ell,\ell]$ is now an interval of
fixed length $c\pi$ which can therefore also be represented as
$[\ell - c\pi,\ell]$. Recall that $q_a$ needs to be built so that
on this interval it matches, through second order, to $e^{\epsilon
x}$ at the left endpoint and similarly to $e^{\epsilon x - \pi}$
at the right endpoint. Toward this end we observe that the
required leading order value on the right is $1$, independent of
$\epsilon$ while on the left the leading order value limits to the
stable value of $e^\pi$ as $\epsilon \to 0$.

Choosing a value $\nu > 0$ that is small with respect to $c\pi$,
we define a compressed \textsl{tanh}-profile that interpolates
between the point $ (x_0, y_0) = (\ell - c\pi + \nu, e^\pi +
\gamma)$ and the point $ (x_1, y_1) = (\ell - \nu, 1 - \gamma)$
and where $\gamma > 0$ is another chosen value required to be
smaller than $1$. (This last requirement will insure that the
positivity claim made at the end of subsection \ref{zipper}
holds.) Explicitly this tanh-profile is given by
\begin{eqnarray*}
  T(x) &=& \frac{e^\pi + 1}{2} + \left(\frac{e^\pi - 1}{2} +
  \gamma\right) \tanh \left( \frac{x - (\ell - c\frac{\pi}{2})}{\left(x-(\ell-\nu)\right)\left(x-(\ell - c\pi +
  \nu)\right)}\right).
\end{eqnarray*}
Note that the profile of $T(x)$ is independent of $\epsilon$. The
only way in which $T$ depends on $\epsilon$ is that this profile
translates uniformly with $\ell$ as $\epsilon$ changes. We will
define $q_a(x) = T(x)$ on the subinterval $[x_0, x_1] = [\ell -
c\pi + \nu,\ell - \nu]$ of $[\ell - c\pi,\ell]$.

Next we will define the piece of $q_a(x)$ on the left that
interpolates between the point $(\ell - c\pi, e^{\pi - \epsilon c
\pi})$ and the point$(x_0, y_0)$. Choose a value $\sigma > 0$ that
is small with respect to $\nu$. Consider the covering of $[\ell -
c\pi,\ell - c\pi + \nu]$ by the two sets $V_1 = [\ell - c\pi,x_0 -
\sigma)$ and $V_2 = (\ell - c\pi + \sigma, x_0]$ and let
$\{\psi_1(x), \psi_2(x)\}$ be a partition of unity subordinate to
this cover which means, in particular, that

\begin{eqnarray*}
   \psi_1 &=& \left\{\begin{array}{cc}
     1 & [\ell - c\pi, \ell - c\pi + \sigma)\\
     0 & (x_0 - \sigma, x_0] \\
\end{array}\right. \\
   \psi_2 &=& \left\{\begin{array}{cc}
     1 & (x_0 - \sigma, x_0] \\
     0 & [\ell - c\pi, \ell - c\pi + \sigma). \\
    \end{array} \right.
\end{eqnarray*}
On $[\ell - c\pi, x_0]$ we define
\begin{eqnarray}\label{qa}
  q_a(x) &=& \psi_1(x) e^{\epsilon x} + \psi_2(x) (e^\pi + \gamma).
\end{eqnarray}
It is straightforward to see that with these choices $q_a(x)$ is
smooth throughout $[\ell - c\pi, \ell - \nu)$ and satisfies the
smooth matching conditions on the left. Moreover, it is clear from
the functions comprising (\ref{qa}) that $q_a$ and its derivatives
remain bounded on $[\ell - c\pi, \ell - \nu)$ as $\epsilon \to 0$.
A similar construction may be made on the right; i.e., on $(\ell -
c\pi + \nu, \ell]$. This completes our description of $q_a(x)$.

The study of the convolution integral in (\ref{shiftpersoln2}) in
the region where $x \in [a\ell, \ell]$ now proceeds similarly to
what was done in (\ref{convoest}) and subsequent formulae. In
particular, the analogous result to (\ref{convoest}) is that for
$x \in (a\ell, \ell)$ and for \textit{times} $t$ of order less
than $1/\epsilon$:

\begin{eqnarray}\label{convoest2}
&&\frac{1}{\ell}\int_{0}^{\ell} w^\epsilon(\xi) \,\,
\vartheta_3\left(\frac{- (x-\xi) + 2\epsilon t}{\ell}, \frac{-4\pi
t}{\ell^2}\right)d\xi \\
\nonumber &=&\frac{1}{\ell}\int_{0}^{\ell} q_a(\xi) \,\,
\vartheta_3\left(\frac{- (x-\xi) + 2\epsilon t}{\ell}, \frac{-4\pi
t}{\ell^2}\right)d\xi\\
 \nonumber &+& \frac{1}{\ell}\int_{a\ell}^{\ell}
\left(w^\epsilon(\xi) - q_a(\xi)\right) \,\,
\vartheta_3\left(\frac{- (x-\xi) + 2\epsilon t}{\ell}, \frac{-4\pi
t}{\ell^2}\right)d\xi \\
\nonumber &=& \frac{1}{\pi}\int_{0}^{\pi} q_a(\frac{z}{\epsilon})
\,\,\vartheta_3\left(\frac{-(h - z) + 2\epsilon^2 t}{\pi},
\frac{-4\epsilon^2 t}{\pi}\right) dz\\
\nonumber &+& \frac{1}{\pi}\int_{0}^{\pi}
\left(w^\epsilon(\frac{z}{\epsilon}) - q_a(\frac{z}{\epsilon})
\right) \,\, \vartheta_3\left(\frac{-(h - z) +
2\epsilon^2 t}{\pi},\frac{-4\epsilon^2 t}{\pi}\right) dz \\
\nonumber &=& q_a(x) + o(\epsilon),
\end{eqnarray}
with $q_a(x)$ here bounded away from zero, independent of
$\epsilon$, by our earlier choice of $\gamma$. Hence the
denominators in the estimates analogous to (\ref{xlogd}) and
(\ref{ylogd}) are under control. In the subsequent formulae the
roles of $e^{\epsilon x}$ and $q_a(x)$ are effectively
interchanged as above and all proceeds as before. The result is
that the energy in $\left( [a\ell,\ell] \times \left\{y <
M_\epsilon\right\}\right)\backslash \mathbf{B}(\delta)$ is
asymptotically bounded like $\mathcal{O}(\epsilon)$.

It thus follows that the total energy of our family of test
functions is uniformly bounded in $\epsilon$.

\section{Lower bound} \label{sec:l_bound}

Following the ideas of Jin and Kohn \cite{JK}, we will prove {\em
ansatz-free} lower bounds for the functional $\mc{F}^{\epsilon}$
by identifying vector fields $\Sigma(\nabla\theta)$ such that
$$
\mc{F}^{\epsilon}[\theta;a,\delta] \geq C^{-1} \left|\iint_{\mc{S}^\epsilon}
  \nabla \cdot \Sigma(\nabla \theta) dx dy \right|.
$$
This allows us to obtain information about the energy
$\mc{F}^\epsilon$ purely in terms of the boundary conditions on
$\theta$. To avoid the proliferation of symbols, here and henceforth,
$C, C', C_1$, {\em etc} denote (finite) constants whose precise value
is unimportant, and different occurrences of the same symbol might
denote different values of the constants. $e_1,e_2,K,K_1$, {\em etc}
denote constants that have the same value in all their occurrences.

\begin{defn}
  A smooth vector function $\Sigma(p,q) = (\Sigma_1, \Sigma_2)$ is
  {\em subordinate to the energy} if
\begin{gather}
  \left|\frac{\partial\Sigma_1(p,q)}{\partial p}\right| +
  \left|\frac{\partial\Sigma_1(p,q)}{\partial q} +
    \frac{\partial\Sigma_2(p,q)}{\partial p}\right| + \left|
    \frac{\partial\Sigma_2(p,q)}{\partial q}\right| \leq C |1 - p^2 -
  q^2|
\label{eq:subordinate}
\end{gather}
for some $C < \infty$.
\end{defn}
If $\Sigma$ is subordinate to the energy, it follows that
\begin{align*}
|\nabla \cdot \Sigma(\nabla \theta)| & \leq  \left|\frac{\partial\Sigma_1(p,q)}{\partial p}\right| |\theta_{xx}| +
  \left|\frac{\partial\Sigma_1(p,q)}{\partial q} +
    \frac{\partial\Sigma_2(p,q)}{\partial p}\right||\theta_{xy}| + \left|
    \frac{\partial\Sigma_2(p,q)}{\partial q}\right||\theta_{yy}| \\
& \leq C|1-\theta_x^2 - \theta_y^2| |\nabla \nabla \theta|,
\end{align*}
where we use the identification $p = \theta_x, q = \theta_y$ and
$|\nabla \nabla \theta|^2 = \theta_{xx}^2 + 2 \theta_{xy}^2 +
\theta_{yy}^2$.
Consequently,
\begin{align*}
\mc{F}^{\epsilon}[\theta;a,\delta] & = \iint_{\mc{S}^\epsilon} \left\{[\nabla \nabla \theta]^2
  + (1 - |\nabla \theta|^2)^2\right\} dx dy - 2 \int_{\partial \mc{S}^\epsilon} \theta_x d \theta_y  \\
& \geq 2 \iint_{\mc{S}^\epsilon} |\nabla \nabla \theta| |1 - |\nabla \theta|^2| dx dy - 2 \int_{\partial \mc{S}^\epsilon} \theta_x d \theta_y \\
& \geq  \frac{2}{C} \left|\iint_{\mc{S}^\epsilon} \nabla \cdot
  \Sigma(\nabla \theta) dx dy \right|  - 2 \int_{\partial \mc{S}^\epsilon} \theta_x d \theta_y \\
& \geq  C^{-1} \left|\iint_{\mc{S}^\epsilon} \nabla \cdot
  \Sigma(\nabla \theta) dx dy \right|.
\end{align*}
In obtaining the last equation, we use the fact
that
$$
\int_{\partial \mc{S}^\epsilon} \theta_x d \theta_y = 0
$$
for the boundary conditions in (\ref{eq:bc}).

\begin{lemma}
  There are constants $e_1,K_1 > 0$ such that for all $\epsilon \in
  (0,1]$,$a \in [0,1]$ and $\delta \in \mathbb{R}$, we have
  $$
  \mc{F}^\epsilon[\theta;a,\delta] \geq \frac{e_1 \epsilon^2}{(1-a)^2} -
  K_1\epsilon^2.
$$
\label{lem:squeeze}
\end{lemma}

\begin{proof}

Let $\phi \geq 0$ be a smooth, compactly supported function such that
\begin{align*}
& \phi(0)  = 1, \\
& \phi(1) < \phi(0),\\
& f(p)  = p \phi(p^2) \text{ has a single maximum at } p=1.
\end{align*}
An explicit example of a function $\phi$ whith these properties is
$$
\phi(p) = \begin{cases} \exp\left[\frac{1}{2} - \frac{1}{(2-p)(p+1)} - \frac{p}{4} \right] & p \in (-1,2) \\
0 & \text{ otherwise }
\end{cases}
$$

 Let
$b = (1-a)/\epsilon$. Define the vector field $\Sigma(p,q)$ by
\begin{align*}
\Sigma_2(p,q) & = p \phi(b^2 p^2) \\
\Sigma_1(p,q) & = - \int_0^q \left[\phi(b^2 (1-\eta^2)) + 2(b^2(1-\eta^2)\phi'(b^2(1-\eta^2)) \right] d\eta.
\end{align*}

Since $\phi$ has compact support, it follows that $\Sigma$ is bounded
on $\mathbb{R}^2$. An explicit calculation shows that the quantities
$\Sigma_{1,p}$ and $\Sigma_{2,q}$ are zero.  Also,
\begin{align*}
 \left|\frac{\partial\Sigma_1(p,q)}{\partial q} + \frac{\partial\Sigma_2(p,q)}{\partial p}\right| & = \left|\phi(b^2 p^2) + 2 b^2 p^2 \phi'(b^2 p^2) - \phi(b^2 (1-q^2)) \right.\\
&\left. - 2b^2(1-q^2)\phi'(b^2(1-q^2)\right| \\
& \leq C b^2 |(1-p^2 - q^2)|
\end{align*}
where
$$
C = \sup_{x,y}\left|\frac{\phi(x) + 2 x \phi'(x) - \phi(y) - 2 y \phi'(y)}{x-y}\right| \leq \sup_z |3 \phi'(z) + 2 z \phi''(z)|
$$
is clearly finite since $\phi$ is compactly supported and twice
differentiable. This proves that $\Sigma$ is subordinate to the
energy.

Since $\nabla \cdot \Sigma(\nabla \theta) = (\Sigma_{2,p} +
\Sigma_{1,q}) \theta_{xy}$ we obtain
\begin{align}
\left|\iint \nabla \cdot \Sigma(\nabla \theta)\, dx dy\right| & \leq Cb^2 \iint| (1 - \theta_x^2 - \theta_y^2) \theta_{xy}| \, dx dy \nonumber \\
& \leq \frac{C b^2}{2} \left[ \iint (1 - \theta_x^2 - \theta_y^2)^2\, dx dy + \iint [\nabla \nabla \theta]^2\, dx dy \right]  \nonumber \\
& = \frac{Cb^2}{2} \mc{F}^\epsilon[\theta;a,\delta] \label{eq:lbnd1}
\end{align}

Integrating by parts, we have
\begin{align*}
\iint \nabla \cdot \Sigma(\nabla \theta)\, dx dy = &
\Sigma_2(\epsilon,\sqrt{1-\epsilon^2}) \frac{\pi}{\epsilon} - \int_0^{a
  \pi/\epsilon} \Sigma_2(0,\theta_y(x,0))
dx \\
&  - \int_{a \pi/\epsilon}^{\pi/\epsilon}
\Sigma_2(\theta_x(x,0),0)dx,
\end{align*}
where the contributions from the boundaries at $x = 0$ and $x =
\pi/\epsilon$ cancel due to the periodicity. By construction,
$\Sigma_2(p,0) = 0$ and $\Sigma_2(p,q)$ has a maximum value at $p =
1/b$.  Consequently,
\begin{align*}
  \int_0^{a \pi/\epsilon}
  \Sigma_2(0,\theta_y(x,0)) dx & = 0 \\
  \int_{a \pi/\epsilon}^{\pi/\epsilon} \Sigma_2(\theta_x(x,0),0) dx &
  \leq \frac{(1 - a) \pi}{\epsilon} \frac{\phi(1)}{b} = \pi \phi(1).
\end{align*}
Also, $\phi(0) = 1$ and $\phi$ is Lipschitz so that
$$
\Sigma_2(\epsilon,\sqrt{1-\epsilon^2}) = \epsilon \phi(b^2
\epsilon^2) = \epsilon \phi((1-a)^2) \geq \epsilon(1 - C' (1-a)^2),
$$
for some finite $C'$. Combining
these estimates with (\ref{eq:lbnd1}), we obtain
\begin{equation}
\mc{F}^\epsilon[\theta;a,\delta] \geq \frac{2 \pi}{C b^2}\left[\phi(0) - \phi(1) - C'(1-a)^2 \right],
\end{equation}
 and rewriting $b$ in terms of $a$ and $\epsilon$ yields
the desired conclusion.
\end{proof}

The above lemma shows that the energy grows without bound as the
quantity $(1-a)$ becomes small. However, we do not have {\em a priori}
control on the size of $(1-a)$. Consequently, to obtain a lower bound
for the energy, we need a complementary estimate which shows that the
energy grows as the quantity $(1-a)$ becomes large.

To prove this result, we first construct a vector field $\Sigma$
subordinate to the energy functional as follows:

Let $\psi \geq 0$ be a smooth, compactly supported function such that
\begin{align*}
& \psi(0)  = 1 \\
& \int_0^{\infty} (1-\xi^2)\psi(\xi^2) d\xi = 0
\end{align*}

We can always construct such a function, given $\chi \geq 0$, a
compactly supported function with $\chi(0) = 1$. Observe that
$$
\int_0^\infty (1 - \xi^2) \chi\left(\frac{\xi^2}{\eta^2}\right) d
\xi = \eta(A_0 - \eta^2 A_1),
$$
where $A_0,A_1 > 0$. Consequently, by an appropriate choice of
$\eta$, we get $\psi(x) = \chi(x/\eta^2)$ with the required
properties.

We define the functions $\zeta(q^2)$ and $\sigma(q^2)$ by
\begin{align}
\zeta(q^2) & = \int_0^q (q - \eta) \psi(\eta^2) d \eta \nonumber \\
\sigma(q^2) & = \int_0^q (q - \eta) (1 - \eta^2) \psi(\eta^2) d \eta
\label{eq:zeta_sigma}
\end{align}
Note that the functions $\zeta$ and $\sigma$ are well defined for
positive values of their arguments, that is the expressions on the
right hand sides of the above equations are even functions of $q$.
From these expressions, we have
\begin{align*}
\frac{\partial^2}{\partial q^2} \zeta(q^2) = \psi(q^2); & \quad \zeta(0) = 0 \\
\frac{\partial^2}{\partial q^2} \sigma(q^2) = (1-q^2)\psi(q^2); & \quad \sigma(0) = 0
\end{align*}
We will also use the same letters $\zeta$ and $\sigma$ to denote
smooth extensions of the functions defined above to all of
$\mathbb{R}$. We will pick extensions such that the supports of
$\zeta$ and $\sigma$ are contained in $[-1,\infty)$.

Let $\varphi \geq 0$ be a compactly supported function and set
\begin{align}
  V(p,q) = & \, \, \varphi(p^2)\left[\sigma(q^2) - p^2\zeta(q^2)\right] \nonumber \\
  & - \int_0^p (p-\xi)\left\{ \sigma(1-\xi^2) \left[
      \frac{\partial^2}{\partial \xi^2} \varphi(\xi^2) \right] -
    \zeta(1-\xi^2) \left[ \frac{\partial^2}{\partial \xi^2} (\xi^2
      \varphi(\xi^2)) \right] \right\}\, d\xi.
\label{eq:defnv}
\end{align}

$V$ is now an even function of $p$ and $q$. Define the vector field
$\Sigma$ by
\begin{equation}
\Sigma(p,q) = \left( -\frac{\partial}{\partial p} V,\frac{\partial}{\partial q} V \right)
\label{eq:def}
\end{equation}
From (\ref{eq:def}) it follows that
$$
\frac{\partial\Sigma_1(p,q)}{\partial q} + \frac{\partial\Sigma_2(p,q)}{\partial p} = 0.
$$
Also,
$$
\left|\frac{\partial\Sigma_2(p,q)}{\partial q}\right| =
\left|\frac{\partial^2 V(p,q)}{\partial q^2}\right| =
|\varphi(p^2)\psi(q^2)||1-p^2-q^2| \leq C_1 |1-p^2-q^2|.
$$
With $C_1 = \sup|\varphi(p^2) \psi(q^2)| < \infty$.  Finally, an
explicit calculation yields
\begin{align}
\left|\frac{\partial\Sigma_1(p,q)}{\partial p}\right| & = \left|\left[
      \frac{\partial^2}{\partial p^2} \varphi(p^2) \right]\left( \sigma(q^2) - \sigma(1-p^2) \right) - \left[
      \frac{\partial^2}{\partial p^2} p^2 \varphi(p^2) \right]\left( \zeta(q^2) - \zeta(1-p^2) \right)\right| \nonumber \\
& \leq C_2 |1 - p^2 - q^2|
\end{align}
where $C_2$ can be bounded in terms of the support of $\varphi$ and the
maximum values of $|\varphi|, |\varphi'|, |\varphi''|, |\zeta'|$ and
$|\sigma'|$. Clearly $\varphi$ and all it's derivatives are uniformly
bounded since it is smooth and compactly supported. From
(\ref{eq:zeta_sigma}), we have
\begin{align*}
\zeta'(q^2) & = \frac{1}{2q}\int_0^q  \psi(\eta^2) d \eta \\
\sigma'(q^2) & = \frac{1}{2q}\int_0^q (1 - \eta^2) \psi(\eta^2) d \eta
\end{align*}
Since $\psi$ is compactly supported, these derivatives vanish as
$q^2\rightarrow \infty$, implying that $\zeta'$ and $\sigma'$ are
bounded for all positive values of the argument. Since $\sigma$ and
$\zeta$ are smooth and are identically zero if their arguments are
sufficiently negative, it follows that $C_2 < \infty$. It thus follows
that the vector field $\Sigma$ is subordinate to the energy
functional.

For future use, let us record a few observations that follow directly
from the construction:
\begin{observation}
$\Sigma_2(p,0) = 0$ since $\Sigma_2$ is an odd function of $q$.
\end{observation}
\begin{observation}
$$
\Sigma_{2,q}(0,q) = V_{qq}(0,q) = \psi(q^2)(1-q^2).
$$
Consequently, the non-degenerate critical points are at $ q = \pm
1$. Differentiating in $q$, we get
$$\Sigma_{2,qq}(0,\pm 1) = \mp 2 \psi(1),$$
so that $\Sigma_2(0,q)$ has a maximum at $q = 1$ and a minimum at $q = -1$.

Finally,
$$\Sigma_2(0,q) = 2 q \sigma'(q^2) = \int_0^q (1-\xi^2) \psi(\xi^2) d
\xi,$$
so that $\Sigma_2(0,1) > 0$ and $\Sigma_2(0,q) \rightarrow 0$
as $q \rightarrow \infty$.
\end{observation}

$M = \int_0^1 (1-\xi^2) \psi(\xi^2) d \xi$ will denote the maximum
value of $\Sigma_2(0,q)$.

\begin{observation}
\begin{align*}
\Sigma_2(\epsilon,\sqrt{1-\epsilon^2}) & = \varphi(\epsilon^2) \int_0^{\sqrt{1-\epsilon^2}} (1- \epsilon^2 -\xi^2) \psi(\xi^2) d \xi \\
& \geq M - K \epsilon^2
\end{align*}
for a constant $K < \infty$. In
obtaining the last line, we use
\begin{align*}
|1-\varphi(\epsilon^2)| & \leq C_1 \epsilon^2 \\
\left|\int_0^{\sqrt{1-\epsilon^2}} \psi(\xi^2) d \xi \right| & \leq C_2 \\
\left|\int_{\sqrt{1-\epsilon^2}}^1 (1-\xi^2) \psi(\xi^2) d \xi \right| & \leq C_3 \epsilon^2
\end{align*}
for some bounded constants $C_1,C_2,C_3$.
\label{obs:estimate}
\end{observation}

\begin{lemma}
  There are constants $e_2,K_2 > 0$ such that for all $\epsilon \in
  (0,1]$, $a \in [0,1]$ and $\delta \in \mathbb{R}$, we have
$$
\mc{F}^\epsilon[\theta;a,\delta] \geq \frac{e_2(1-a)}{\epsilon} - K_2 \epsilon.
$$
\label{lem:extend}
\end{lemma}

\begin{remark}
  For the case $a = 0$, corresponding to the self-dual minimizers,
  this estimate captures the right scaling of the minimum energy as
  $\epsilon \rightarrow 0$.
\end{remark}

\begin{proof}
  The proof of the lemma follows from estimating a lower bound for the
  functional $\mc{F}^\epsilon[\theta;a,\delta]$ using the vector field
  $\Sigma$ that we constructed above.

For the vector field $\Sigma$ we have
\begin{align*}
\iint \nabla \cdot \Sigma(\nabla \theta)\, dx dy = &
\Sigma_2(\epsilon,\sqrt{1-\epsilon^2}) \frac{\pi}{\epsilon} - \int_0^{a
  \pi/\epsilon} \Sigma_2(0,\theta_y(x,0)) dx \\
& - \int_{a
  \pi/\epsilon}^{\pi/\epsilon} \Sigma_2(\theta_x(x,0),0) dx.
\end{align*}
As before, the contributions from the boundaries at $x = 0$ and $x
= \pi/\epsilon$ cancel due to the periodicity. By construction,
$\Sigma_2(p,0) =0$ and $\Sigma_2(0,q)$ has a maximum value $M$ at $q =
1$.  Consequently,
$$
\int_0^{a \pi/\epsilon}
\Sigma_2(0,\theta_y(x,0)) dx \leq \frac{M a \pi}{\epsilon}
$$
From observation~\ref{obs:estimate}, we obtain
$$
\Sigma_2(\epsilon,\sqrt{1-\epsilon^2}) \frac{\pi}{\epsilon} \geq \frac{M \pi}{\epsilon} - K \pi  \epsilon,
$$
Since $\Sigma$ is subordinate to the energy,
\begin{equation}
\mc{F}^\epsilon[\theta;a,\delta] \geq \frac{M \pi}{C \epsilon}\left[1 - a  - K \epsilon^2 \right],
\end{equation}
which yields the desired conclusion.
\end{proof}

We can now prove theorem~\ref{thm:l_bound} using
lemma~\ref{lem:squeeze} and lemma~\ref{lem:extend}.

\begin{proof}

  Let $b$ denote the quantity $(1-a)/\epsilon$. From
  lemma~\ref{lem:squeeze} we get
  $$
  \mc{F}^\epsilon[\theta;a,\delta] \geq \frac{e_1}{b^2} - K_1 \epsilon^2
  \geq 3 \left(e_1 e_2^2\right)^{1/3} - 2 e_2 b - K_1 \epsilon^2
$$
where the last inequality comes from linearizing the convex
function $e_1 b^{-2}$ at $b = (e_1/e_2)^{1/3}$. Combining this estimate
with the conclusion of lemma~\ref{lem:extend}, we get
$$
\mc{F}^\epsilon[\theta;a,\delta] \geq \max\left(3 \left(e_1
    e_2^2\right)^{1/3} - 2 e_2 b - K_1 \epsilon^2, e_2 b - K_2 \epsilon
\right) \geq (e_1 e_2^2)^{1/3} - \frac{2 K_2 \epsilon +K_1 \epsilon^2
}{3}.
$$
If we set
$$
\epsilon_* = \min\left( \frac{(e_1 e_2^2)^{1/6}}{\sqrt{K_1}},
  \frac{(e_1 e_2^2)^{1/3}}{2 K_2}\right),
$$
for all $\epsilon < \epsilon_*$, all $a \in [0,1]$ and all $\theta$
satisfying the boundary conditions in (\ref{eq:bc}), we have
$$
\mc{F}^{\epsilon}[\theta;a,\delta] \geq \frac{(e_1 e_2^2)^{1/3}}{3} \equiv E_1.
$$

Combining the upper bound
$\mc{F}^{\epsilon}[\theta^\epsilon;a^\epsilon,\delta^\epsilon] \leq E_0$ in
theorem~\ref{thm:u_bound} with the lower bounds for $\mc{F}^\epsilon$
in lemma~\ref{lem:squeeze} and lemma~\ref{lem:extend}, it follows that
for
$$\epsilon < \min\left(\sqrt{\frac{E_0}{K_1}},\frac{E_0}{K_2}\right),$$
we have
$$
\sqrt{\frac{e_1}{2E_0}} < \frac{1-a^\epsilon}{\epsilon} < \frac{2 E_0}{e_2}.
$$
Consequently,
$$
1 - \alpha_2 \epsilon <  a^\epsilon < 1 - \alpha_1 \epsilon,
$$
for sufficiently small $\epsilon$ with $\alpha_1 =
\sqrt{e_1/(2E_0)}$ and $\alpha_2 = 2 E_0/e_2$.
\end{proof}

\section{Numerical Results} \label{sec:results}

In this section, we will present the results of numerical
simulations that illustrate and clarify our analysis of the energy
and also the structure of the minimizers for the regularized
Cross-Newell energy $\mathcal{F}^\epsilon$ within the class of
functions given by (\ref{eq:bc}).

For our numerical simulations, we restrict ourself to the finite
domain, $\mc{R}^{\epsilon} = \{(x,y) \, |\, 0 \leq x \leq
l = \pi/\epsilon, 0 \leq y \leq L\}$, where $L \gg 1$ is a length scale
much larger than the typical wavelength of the pattern. The boundary
conditions in (\ref{eq:bc}) which are appropriate for the
semi-infinite strip $\mc{S}^\epsilon$ are modified for the
finite domain as follows --
\begin{gather}
\theta(x,0) = 0 \quad \text{ for } 0 \leq x < a l; \nonumber \\
\theta_y(x,0) = 0 \quad \text{ for } a l \leq x < l; \nonumber \\
\theta(x,y) - \epsilon x \text{ is periodic in $x$ with period $l$ for each } y \in [0,L]; \nonumber \\
\theta(x,L) = \left[\epsilon x + \sqrt{1 - \epsilon^2} L + \delta\right]
 \label{eq:num_bc}
\end{gather}

  It is rather straightforward to show that there exist $\theta^\epsilon \in
  H^2(\mc{R}^\epsilon)$ satisfying (\ref{eq:num_bc}) for an
  $a^\epsilon \in [0,1)$ and $\delta^\epsilon \in \mathbb{R}$
  minimizing the functional
$$
\mc{F}^\epsilon[\theta;a,\delta] = \iint_{\mc{R}^\epsilon} \left\{[
  \Delta \theta]^2 + (1 - |\nabla \theta|^2)^2 \right\} dx dy.
$$

The existence of a minimizer is immediate from the following lemma:

\begin{lemma}
  Let $0 \leq a < 1$, and $\delta \in \mathbb{R}$ be given. $\rho_j \in
  L^2(\mc{R}^\epsilon)$ is a sequence of functions that converges
  weakly to zero. $H^2_{per}$ denotes the completion of periodic (in
  $x$) functions on $\mc{R}^\epsilon$ with respect to the $H^2$ norm.
  If $\theta_j \in H^2_{per}(\mc{R}^\epsilon)$ is a sequence
  satisfying (in the sense of trace)
\begin{gather}
  \Delta \theta_j = \rho_j \nonumber \\
  \theta_j(x,0) = 0 \quad \text{ for } 0 \leq x < a l; \nonumber \\
  \partial_y \theta_j(x,0) = 0 \quad \text{ for } a l \leq x < l; \nonumber \\
  \theta(x,L) = 0
\end{gather}
it follows that, up to extraction of a subsequence and
relabelling, we have $\nabla \theta_j \rightarrow 0$ in
$L^4(\mc{R}^\epsilon,\mathbb{R}^2)$.
\end{lemma}

\begin{proof} Elliptic regularity along with the given boundary conditions
  implies that the sequence $\theta_j$ is bounded in
  $H^2(\mc{R}^\epsilon)$. The compactness of the embedding
  $H^2(\mc{R}^\epsilon) \hookrightarrow W^{1,4}(\mc{R}^\epsilon)$
  \cite{Evans} proves the lemma.
\end{proof}

If $\tilde{\theta}_j$ is an infimizing sequence for
$\mc{F}^\epsilon[\theta;a,\delta]$ subject to the boundary conditions
in (\ref{eq:num_bc}), then let $\theta_j = \tilde{\theta}_j -\varphi$,
where $\varphi$ is a smooth function on $\mc{R}^\epsilon$ satisfying
the boundary conditions in (\ref{eq:num_bc}). It then follows from the
form of $\mc{F}^\epsilon$ and the fact that $\tilde{\theta}_j$ is
infimizing that $\Delta \theta_j$ is a bounded sequence in $L^2$, and
so converges weakly to a limit $\rho^*$. Applying the compactness
result of the preceding lemma with reference to the sequence $\rho_j =
\Delta \theta_j - \rho^*$, we obtain the existence of a minimizier for
the functional $\mc{F}^\epsilon[\theta;a,\delta]$ for a fixed $a$ and
$\delta$.

Note that, for a given $a$ it is easy to construct smooth
transformations $\psi_t : \mc{R}^\epsilon \rightarrow \mc{R}^\epsilon$
such that $\psi_0$ is the identity, if $\theta$ satisfies the boundary
conditions in (\ref{eq:num_bc}), then $\theta \circ \psi_t$ satisfies
the same boundary conditions with the fraction of the boundary with a
Dirichlet boundary condition equaling $a(1+t)$. Further, the energy
$\mc{F}^\epsilon[\theta \circ \psi_t;a(1+t),\delta]$ is a smooth
function of $t$ for sufficiently small $t$. A standard argument now
implies that, for a given $\delta$ the map
$$
a \mapsto \inf_{\theta} \mc{F}^\epsilon[\theta;a,\delta]
$$
is continuous for $a \in (0,1)$. A similar argument shows that the
map is also continuous at $a = 0$. In Lemma~\ref{lem:squeeze} we
showed that $\liminf_{a \rightarrow
  1} \mc{F}^\epsilon[\theta;a,\delta] = \infty$. Combining these
results, we see that
$$
\inf_{a,\theta} \mc{F}^\epsilon[\theta;a,\delta] = \inf_a \left[\inf_{\theta} \mc{F}^\epsilon[\theta;a,\delta]\right].
$$

We now consider variations $\theta \rightarrow \theta_t = \theta + t
\chi(y/L)$, where $\chi$ is a smooth function vanishing identically on
$[0,1/3]$ and equal to 1 on $[2/3,1]$. The functions $\theta_t$
satisfy the boundary conditions in (\ref{eq:num_bc}), except the
asymptotic phase shift is given by $\delta + t$.  A similar argument
as above shows that the map
$$
\delta \mapsto \inf_{\theta} \mc{F}^\epsilon[\theta;a,\delta]
$$
is continuous and it is easy to see that $
\mc{F}^\epsilon[\theta;a,\delta] \rightarrow \infty$ as $\delta
\rightarrow \pm \infty$. In particular, this proves the existence
of an optimal $\delta$, and combining with the results from above,
we see that the minimizer $\theta^{\epsilon}, a^\epsilon,
\delta^{\epsilon}$, can be obtained by successive minimization in
each of the factors.


This suggests the following discretization for the functional
$\mc{F}^\epsilon$, which should converge as the grid spacings $\eta,\zeta
\rightarrow 0$. We define a grid by $x_i = i \eta, i = 0,1,2,\ldots,m-1,
y_j = j \zeta, j = 0,1,2,\ldots,n$, where $\eta = l/m, \zeta = L/n$. The
discretization of the test function $\theta(x,y)$ is
$$
\theta_{i,j} = \theta(i \eta,j \zeta)
$$
We define the difference operator $\delta^{\pm}_x$ by
$$
(\delta^{\pm}_x \theta)_{i,j} = \pm\frac{\theta_{i \pm 1,j} - \theta_{i,j}}{\eta}
$$
with similar definitions for $\delta^{\pm}_y$. In terms of the
discretization, the boundary conditions are
\begin{gather}
\theta_{i,0} = 0 \quad \text{ for } 0 \leq i < k; \nonumber \\
\delta_y^+\theta_{i,0} = 0 \quad \text{ for } k \leq i < m; \nonumber \\
\theta_{m,j} = \theta_{0,j} + \pi \quad j = 0,1,2,\ldots,n \\
\theta_{i,n} = \frac{\pi i}{m} + \sqrt{1-\epsilon^2} L + \delta \quad i = 0,1,2,\ldots m-1
 \label{eq:discr_bc}
\end{gather}
and the Energy functional is discretized as
$$
\mc{F}^\epsilon \approx \eta \zeta \sum_{i = 0}^{m-1} \sum_{j = 0}^{n}
\left[ (\delta_x^+ \delta_x^-  + \delta_y^+ \delta_y^-) \theta_{i,j}\right]^2 + \left[ \frac{(\delta_x^+ \theta)_{i,j}^2 +(\delta_x^- \theta)_{i,j}^2 +(\delta_y^+ \theta)_{i,j}^2 +(\delta_y^- \theta)_{i,j}^2 }{2} - 1 \right]^2
$$
Computing this functional requires assigning values for
$\theta_{i,j}$ with $i = -1, j = -1$ and $j = n+1$. The values for $i
= -1$ are obtained form the shift-periodicity of $\theta$ by
$\theta_{-1,j} = \theta_{m-1,j} - \pi$. The values for T $j = n+1$ are
assigned using the Dirichlet boundary condition $\theta_{i,n+1} =
\frac{\pi i}{m} + \sqrt{1-\epsilon^2} (L+\zeta) + \delta$. This functional is
minimized using MATLAB's conjugate-gradient minimization.

Fig.~\ref{fig:hyster} shows the results from minimizing the RCN
energy over the pattern $\theta$ and also the phase shift
$\delta$, for different values of $\epsilon$, and for a range of
values of $a \approx k/m$. The results do indeed suggest that the
(partial) minimization with respect to the pattern and the
asymptotic phase yields a functional that {\em depends
continuously} on $a$. Further, this functional has first-order
(discontinuous) phase transition at a bifurcation value
$\epsilon^*$, below which the global minimizer has $a \neq 0$.

\begin{figure}[htbp]

\centerline{\includegraphics[width = 0.7\hsize,viewport = 0 0 675 478,clip]{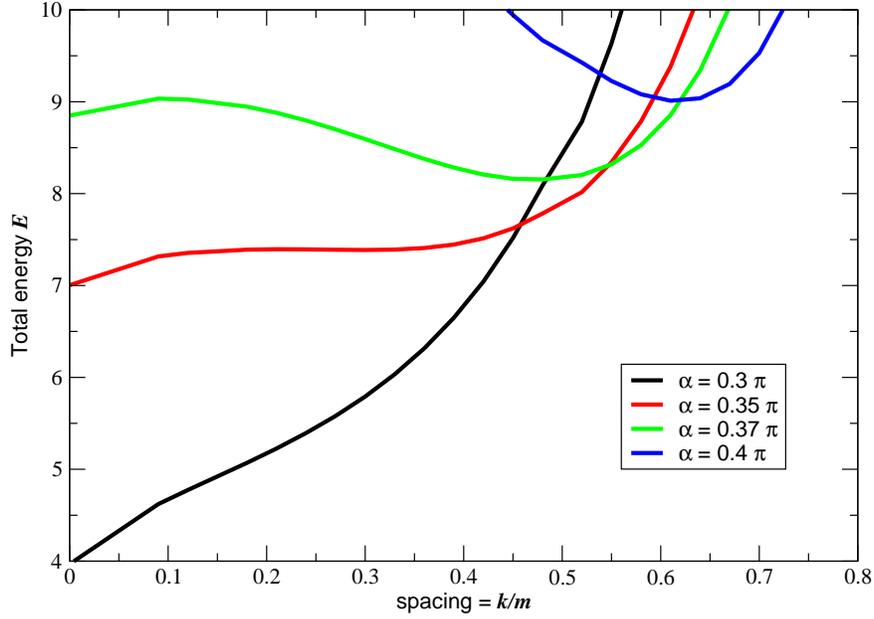}}
\caption{Minimum energy as a function of the separation between the convex and the concave disclinations.}
\label{fig:hyster}
\end{figure}

Fig.~\ref{fig:energy} shows the energy of the minimizer (minimizing
over the pattern, asymptotic phase and the parameter $a$) as a
function of $\epsilon$. Note that the minimum energy is a
non-differentiable function of $\epsilon$, as one would expect for a
first-order phase transition.

\begin{figure}[htbp]

\centerline{\includegraphics[width = 0.7\hsize,viewport = 0 0 675
478,clip]{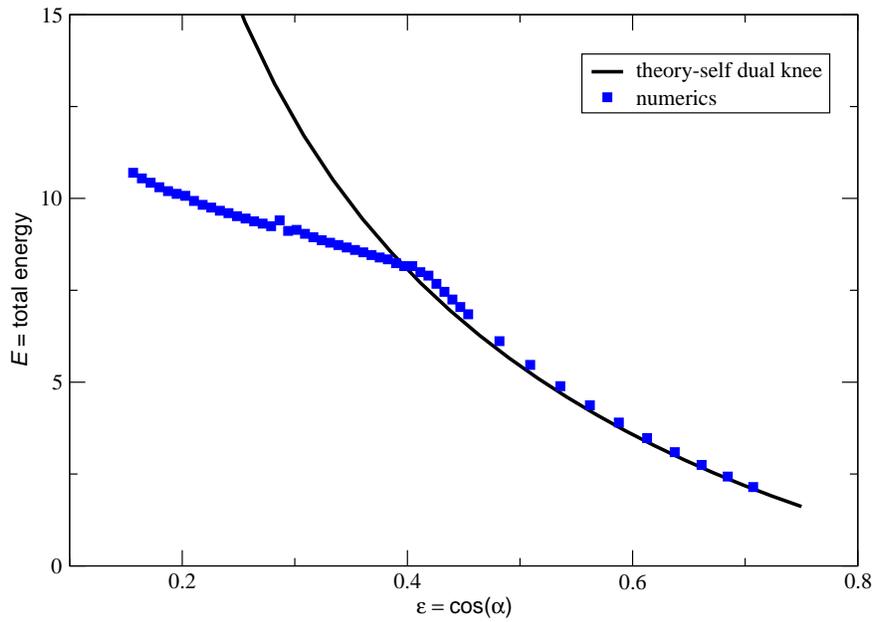}} \caption{Global minimum of the
regularized Cross-Newell energy.} \label{fig:energy}
\end{figure}

Figure~\ref{fig:patterns} show the numerically obtained minimizing
patterns at various values of $\epsilon$. Note that, for sufficiently
large $\epsilon$, the minimizers are the knee-solutions
(\ref{eq:chevrons}) with $a = 0$, whereas for sufficiently small
$\epsilon$, the minimizers have convex-concave disclination pairs, and
have $a \neq 0$. 

\begin{figure}[htbp]

\includegraphics[angle=90]{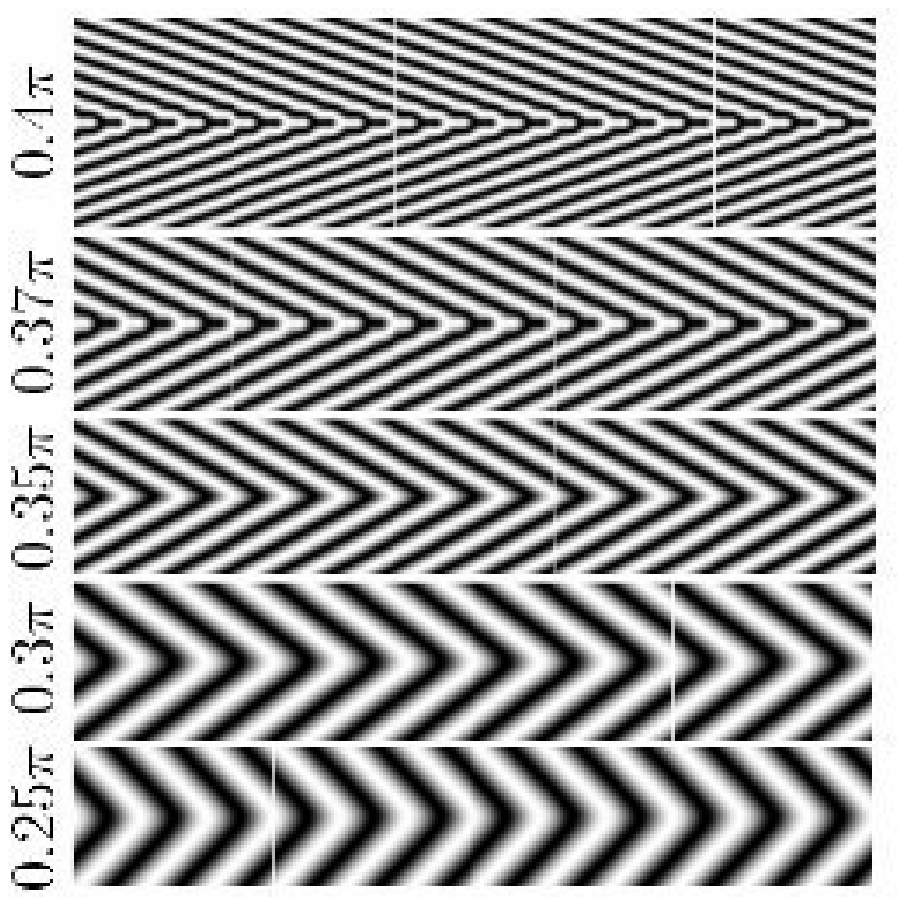}



\label{fig:patterns}
\caption{ The Cross-Newell zippers. These are numerically obtained
minimizing patterns for various choices of the asymptotic angle
$\alpha$.  Note that $\epsilon = \cos \alpha$. The bifurcation from
the knee solution to solutions with disclinations occurs between
$\alpha = 0.35 \pi$ and $\alpha = 0.37 \pi$.
}
\end{figure}

\bigskip

\thanks{\noindent \textbf{Acknowledgements:} N. M. Ercolani was supported in part by NSF grant DMS-0073087;
S.C. Venkataramani was supported in part by an NSF CAREER Award DMS--0135078.}

\bibliographystyle{amsplain}

\begin{thebibliography}{7}
\bibitem{AS} M. Abramowitz and I. Stegun. Handbook of Mathematical
Functions,
U.S. Govt. Printing Office, Washington, D.C., 1972.
\bibitem{AG} P. Aviles, Y. Giga, On Lower Semicontinuity of a
Defect Energy obtained by a Singular Limit of the Ginzburg-Landau
Type Energy for Gradient Fields, Proc. Roy. Soc. Edinburgh,
\textbf{129A} (1999) 1-17.
\bibitem{CN} M.C. Cross, A.C. Newell, Convection Patterns in Large
Aspect Ratio Systems, Physica D \textbf{10} (1984) 299-328.
\bibitem{EINP} N.M. Ercolani, R. Indik, A.C. Newell, T. Passot,
The Geometry of the Phase Diffusion Equation, J. Nonlinear Sci.
\textbf{10} (2000) 223-274.
\bibitem{EINP2} N.M. Ercolani, R. Indik, A.C. Newell, T. Passot,
Global Description of Patterns Far from Onset: A Case Study,
Physica D \textbf{184} (2003) 127-140.
\bibitem{ET} N.M. Ercolani, M. Taylor, The Dirichlet-to-Neumann
Map, Viscosity Solutions to Eikonal Equations, and the Self-Dual
Equations of Pattern Formation, Physica D \textbf{196} (2004)
205-223.
\bibitem{Evans} L. C. Evans. Partial Differential Equations,
American Mathematical Society, Providence, RI, 2002.
\bibitem{JK} W. Jin, R. Kohn, Singular Perturbation and the Energy
of Folds, J. Nonlinear Sci. \textbf{10} (2000) 355-390.
\bibitem{WW} E.T. Whittaker and G.N. Watson. A Course in Modern
Analysis, Cambridge University Press, Cambridge, England, 1990.
\end{thebibliography}

\end{document}